\documentclass[a4paper,11pt,oneside,reqno]{amsart}

\usepackage{hyperref}
\usepackage[headinclude,DIV13]{typearea}
\areaset{15.1cm}{25.0cm}
\usepackage[utf8]{inputenc}
\usepackage[T1]{fontenc}
\usepackage{amsthm,amsmath,txfonts,amssymb,bbm,hyperref,mathtools,csquotes,thmtools,verbatim,xfrac,bigints}
\usepackage{chngcntr}
\usepackage{textcomp}
\usepackage{tikz}
\usepackage{pgfplots}
\usepackage{multicol}
\usepackage[title,titletoc]{appendix}

\usepackage[UKenglish]{babel}

\theoremstyle{plain}
\newtheorem{thm}{Theorem}[section]
\newtheorem{cor}[thm]{Corollary}
\newtheorem{prop}[thm]{Proposition}
\newtheorem{lemma}[thm]{Lemma}
\newtheorem{lemma_appendix}[thm]{Lemma}
\theoremstyle{definition}
\newtheorem{definition}[thm]{Definition}
\theoremstyle{remark}
\newtheorem{rem}[thm]{Remark}

\def\paragraph#1{\noindent \textbf{#1}}

\numberwithin{equation}{section}

\DeclarePairedDelimiter\ceil{\lceil}{\rceil}

\begin{document}
	\title[Extremes of the 2d scale-inhomogeneous discrete Gaussian free field]{Extremes of the 2d scale-inhomogeneous discrete Gaussian free field: Sub-leading order and exponential tails}
	\author{Maximilian Fels}
	\address{M. Fels\\Institut f\"ur Angewandte Mathematik\\
		Rheinische Friedrich-Wilhelms-Universität\\ Endenicher Allee 60\\ 53115 Bonn, Germany }
	\email{fels@iam.uni-bonn.de, wt.iam.uni-bonn.de/maximilian-fels}
	\thanks{M.F. is funded by the Deutsche Forschungsgemeinschaft (DFG, German Research Foundation) - project-id 211504053 - SFB 1060 and Germany’s Excellence Strategy – GZ 2047/1, project-id 390685813 – “Hausdorff Center for Mathematics” at Bonn University.}
	\begin{abstract}
		This is the first of a three paper series in which we present a comprehensive study of the extreme value theory of the scale-inhomogeneous discrete Gaussian free field. This model was introduced by Arguin and Ouimet in \cite{MR3541850} in which they computed the first order of the maximum. In this first paper we establish tail estimates for the maximum value, which allow to deduce the log-correction to the order of the maximum and tightness of the centred maximum. Our proofs are based on the second moment method and Gaussian comparison techniques.
	\end{abstract}

	\maketitle
%\section*{Abstract}
%\input{abstract.tex}
%\section{Introduction}
%\input{introduction.tex}
%\section{The time-inhomogeneous branching random walk}
%\input{scale_dep_brw.tex}
%\section{The time-inhomogeneous branching Brownian motion}
%\input{bbm.tex}
\section{Introduction}
In recent years, so-called  log-correlated (Gaussian) processes have received considerable attention, see e.g. \cite{MR3594368,MR3911893,2016arXiv160600510B,2014arXiv1412.5975B,2015arXiv150304588D,fyodorov1,MR3361256}. One of the reasons for this is that their correlation structure becomes relevant for the properties of the extremes of the processes. Some prominent examples that fall into this class are branching Brownian motion (BBM), the two-dimensional discrete Gaussian free field ( $2$d DGFF), local maxima of the randomised Riemann zeta function on the critical line and cover times of Brownian motion on the torus. 
The $2$d DGFF is one of the well understood non-hierarchical log-correlated fields (see \cite{MR3509015,2016arXiv160600510B,MR1880237,MR3433630}). For simplicity, consider the $2$d DGFF on a square lattice box of side length $N$. It turns out that the maximum can be written as a first order term which is proportional to the logarithm of the volume of the box, a second order correction which is proportional to the logarithm of the first order and stochastically bounded fluctuations. If one considers an uncorrelated Gaussian field on the same box with identical variances, a simple computation shows that the first order of the maximum coincides with the one of the DGFF, whereas the constant in front of the second order correction differs. In \cite{MR3541850}, Arguin and Ouimet introduced the scale-inhomogeneous $2$d DGFF, the analogue  model of variable speed BBM \cite{MR3361256}, which allows to consider different variance profiles. They determined the first order of the maximum. In this paper we continue the study of the maximum, find tail estimates on the maximum value which allow us to deduce the second order correction and tightness of the centred maximum. In the other two papers in this series, we prove, in the regime of weak correlations, convergence of the centred maximum \cite{paper2} and convergence of the extremal process \cite{paper3}. Both are joint work with Hartung.
\subsection{The $2$d discrete Gaussian free field}
Let $V_N \coloneqq ([0,N) \cap \mathbbm{Z} )^2$. The interior of $V_N$ is defined as $V_N^o \coloneqq ([1,N-1] \cap \mathbbm{Z})^2$ and the boundary of $V_N$ is denoted by $ \partial{V}_N \coloneqq V_N \setminus V_N^o$. Moreover, for points $u,v \in V_N$ we write $u \sim v$, if and only if $\|u-v \|_2 =1$, where $\| . \|_2$ is the Euclidean norm.
Let $\mathbb{P}_u$ be the law of a SRW $\{W_k\}_{k \in \mathbbm{N}}$ starting at $u \in {\mathbbm{Z}}^2$. The normalised Green kernel is given by
\begin{equation}
G_{V_N}(u,v) \coloneqq \frac{\pi}{2} \mathbbm{E}_u \left[\sum_{i=0}^{\tau_{\partial{V}_N-1}}\mathbbm{1}_{\{W_i =v\}}\right], \text{ for } u,v \in V_N.
\end{equation}
Here, $\tau_{\partial{V}_N}$ is the first hitting time of the boundary $\partial{V}_N$ by $\{W_k\}_{k \in \mathbbm{N}}$.
For $\delta >0$, we set $V^{\delta}_N\coloneqq (\delta N, (1-\delta)N)^2\cap \mathbb{Z}^2$. By \cite[Lemma 2.1]{MR2243875}, we have for $\delta \in (0,1)$ and $u,v\in V^\delta_N$,
\begin{align}\label{equation:correlation_dgff}
	G_{V_N}(u,v)= \log N - \log\left( \|u-v\|_2\vee 1\right) + O(1).
\end{align}
\begin{definition}\label{definition:DGFF}
	The $2$d discrete Gaussian free field (DGFF) on $V_N$, $\phi^N \coloneqq \{\phi_v^N \}_{v \in V_N}$, is a centred Gaussian field with covariance matrix $G_{V_N}$ and entries  $G_{V_N}(x,y)=\mathbbm{E}[\phi^N_x \phi^N_y],$ for $x,y \in V_N.$
\end{definition}
From \autoref{definition:DGFF} it follows that $\phi_v^N = 0$ for $v \in \partial{V}_N,$ i.e. we have Dirichlet boundary conditions.
\subsection{The $2$d scale-inhomogeneous discrete Gaussian free field}\label{sec:scale_GFF}
\begin{definition}{($2$d scale-inhomogeneous discrete Gaussian free field).}\label{definition:scale_inh_DGFF} \\
Let $\phi^N= \{\phi^N_v \}_{v\in V_N}$ be a $2$d DGFF on $V_N$.
For $v=(v_1,v_2)\in V_N$, let $[v]^N_\lambda$ be the box of side length $N^{1-\lambda}$ centred at $v$, namely
\begin{align}
	[v]_{\lambda}\equiv[v]^N_{\lambda} \coloneqq& \left( \left[v_1-\frac{1}{2}N^{1-\lambda},v_1+\frac{1}{2}N^{1-\lambda}\right.\right] \times \left[ \left.v_2-\frac{1}{2}N^{1-\lambda},v_2+\frac{1}{2}N^{1-\lambda}\right]\right) \cap V_N
\end{align}
and set $[v]^N_{0} \coloneqq V_N$ and $[v]^N_{1} \coloneqq \{v\}$. We denote by $[v]^o_{\lambda}$ the interior of $[v]_\lambda$. Let $\mathcal{F}_{\partial{[v]_{\lambda}} \cup [v]_{\lambda}^c } \coloneqq \sigma\left(\{ \phi^N_v, v \notin [v]_{\lambda}^o \}\right)$ be the $\sigma-$algebra generated by the random variables outside $[v]_{\lambda}^o$.
We define $\phi^N_v(\lambda)$ by conditioning on the DGFF outside the box $[v]^N_\lambda$, i.e.
\begin{align}\label{equation:condition_dgff__lambda_env}
	\phi^N_v(\lambda)=\mathbb{E}\left[\phi^N_v | \mathcal{F}_{\partial{[v]_{\lambda}} \cup [v]_{\lambda}^c } \right],\quad  \lambda \in [0,1].
\end{align}
We denote by $\nabla \phi^N_v(\lambda)$ the derivative $\partial_{\lambda} \phi^N_v(\lambda)$ of the DGFF at vertex $v$ and scale $\lambda$.
Further, let $s\mapsto \sigma(s)$ be a non-negative function such that $\mathcal{I}_{\sigma^2}(\lambda)\coloneqq\int_{0}^{\lambda}\sigma^2(x)\mathrm{d}x$ is a non-decreasing function on $[0,1]$ with $\mathcal{I}_{\sigma^2}(0)=1$ and $\mathcal{I}_{\sigma^2}(1)=1$.
Then the $2$d scale-inhomogeneous DGFF on $V_N$ is a centred Gaussian field $ \psi^N \coloneqq \{\psi^N_v \}_{v \in V_N}$ defined as
\begin{align}\label{equation:1.5}
	\psi^N_v\coloneqq \int_{0}^{1} \sigma(s) \nabla \phi^N_v(s) \mathrm{d}s.
\end{align}
\quad In this paper, we consider the case when $\sigma$ is a right-continuous step function taking $M\in \mathbb{N}$ values.
Thus, there are variance parameters $(\sigma_1,\dotsc,\sigma_M) \in [0,\infty)^M$ and scale parameters $(\lambda_1,\dotsc,\lambda_M) \in (0,1]^M$ with $0 \eqqcolon \lambda_0 < \lambda_1 \dotsc < \lambda_M \coloneqq 1,$ such that
\begin{equation} \label{parameters1}
\sigma (s) = \sum_{i=1}^{M}\sigma_i \mathbbm{1}_{[\lambda_{i-1},\lambda_i)}(s), \text{ } s \in [0,1].
\end{equation}
In this case, the scale-inhomogeneous 2d DGFF or $2$d $(\sigma, \lambda)-$DGFF in \eqref{equation:1.5} takes the form
	\begin{equation}
		\psi^N_v = \sum_{i=1}^{M} \sigma_i(\phi^N_v (\lambda_i)-\phi^N_v (\lambda_{i-1})).
	\end{equation}
Similarly to \eqref{equation:condition_dgff__lambda_env}, we set for $v\in V_N$ and $\lambda \in [0,1]$,
\begin{align}
	\psi^N_v(\lambda)\coloneqq \mathbb{E}\left[\psi^N_v \bigg| \mathcal{F}_{\partial{[v]_{\lambda}} \cup [v]_{\lambda}^c }\right].
\end{align}
Next, we compute the covariances of $\{ \psi^N_v \}_{v\in V_N}$. We fix $\delta\in (0,1/2)$ and $\lambda \in (4\delta/ \log N,1/ \sqrt{\log N})$. For $N\in \mathbb{N}$ and $v,w\in V_N$, set $q_N(v,w)\coloneqq \frac{\log N- \log \|v-w\|_2}{\log N}$. For $v,w\in V_N^\delta$, we write $\mathbb{E}\left[\psi^N_v \psi^N_w\right]=\mathbb{E}\left[\left(\psi^N_v-\psi^N_v(\lambda)\right)\psi^N_w+\psi^N_v(\lambda)\psi^N_w \right]$.
By choice of $\delta$ and $\lambda$, it holds that $[v]^N_\lambda \cap \partial V_N = \emptyset$ and $[w]^N _\lambda \cap \partial V_N = \emptyset$. Therefore, we may deduce as in \cite[(A.41), (A.42)]{MR3731796},
\begin{align}\label{equation:1.9}
	\mathbb{E}\left[\left(\psi^N_v-\psi^N_v(\lambda)\right)\psi^N_w \right]=\left[ \mathcal{I}_{\sigma^2}\left(q_N(v,w)\right)- \mathcal{I}_{\sigma^2}\left(\min\left\{\lambda,q_N(v,w)\right\}\right)\right]\log N 
	+O(\sqrt{\log N}),
\end{align}
and
\begin{align}\label{equation:1.10}
	\left|\mathbb{E}\left[\psi^N_v(\lambda)\psi^N_w \right]\right|\leq O(\sqrt{\log N}).
\end{align}
Using \eqref{equation:1.9} and \eqref{equation:1.10}, we obtain for $v,w\in V_N^\delta$,
\begin{align}
	\mathbb{E}\left[\psi^{N}_v\psi^N_w\right]= \log N \mathcal{I}_{\sigma^2}\left(\frac{\log N - \log\left( \|v-w\|_2\vee 1\right)}{\log N}\right) +O(\sqrt{\log(N)}).
\end{align}
\end{definition}

\section{Main result}
The main result of this paper are tail estimates for the maximum of the scale-inhomogeneous 2d DGFF when there are finitely many scales. As simple consequences, we deduce the correct second order correction and tightness of the centred maximum.
We start with some notation.
Let $\hat{\mathcal{I}}_{\sigma^2}(s)$ be the concave hull of $\mathcal{I}_{\sigma^2}(s)$.
There exists a unique non-increasing, right-continuous step function $ s \rightarrow \bar{\sigma}(s)$, which we call 'effective variance', such that
\begin{equation}\label{eq:int_hat}
	\hat{\mathcal{I}}_{\sigma^2}(s)=\int_{0}^{s} {\bar{\sigma}}^2(r)\mathrm{d}r\eqqcolon \mathcal{I}_{{\bar{\sigma}}^2}(s) \quad \text{for all } s \in [0,1].
\end{equation}
The points where $\bar{\sigma}$ jumps on $[0,1]$ we call 
\begin{equation}
	0 \eqqcolon \lambda^0 < \lambda^1< \dotsc < \lambda^m \coloneqq 1,
\end{equation}
where $ m\leq M.$ To be consistent with previous notation (cf.\eqref{parameters1}), we write $\bar{\sigma}_l \coloneqq \bar{\sigma}(\lambda^{l-1}).$
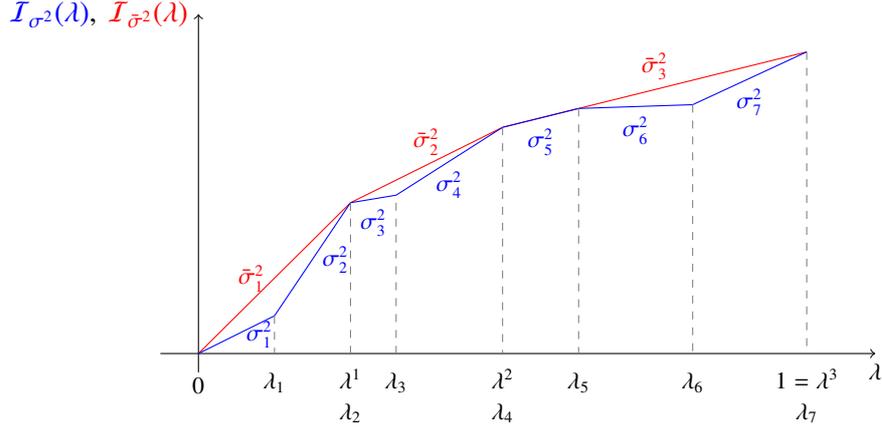
\begin{figure}
	\begin{tikzpicture}[xscale=2]
	%\draw[xstep=1cm, ystep=1cm, ,color=lightgray] (0,0) grid (4,4);
	
	\draw[->] (-0.25,0)--(4.45,0) node[below]{\footnotesize $\lambda$};
	
	%\foreach \x in {1,...,4}
	%\draw[shift={(\x,0)},color=black] (0pt,2pt) -- (0pt,-2pt) node[below]
	%{\footnotesize $\x$};
	
	\draw[->] (0,-0.25) -- (0,4.5) node[left] {$\color{blue}\mathcal{I}_{\sigma^2}(\lambda)$,\, $\color{red}\mathcal{I}_{\bar{\sigma}^2}(\lambda)$};
	
	%\foreach \y in {1,...,4}
	%\draw[shift={(0,\y)},color=black] (2pt,0pt) -- (-2pt, 0pt) node[left]
	%{\footnotesize $\y$};
	
	\draw[color=black] (0pt, -5pt) node[below] {\footnotesize $0$};
	
	\draw[color=red] (0,0) -- (1,2) node[midway,left] {\footnotesize $\bar{\sigma}_1^2$};
	\draw[color=red] (1,2) -- (2,3) node[midway,above] {\footnotesize $\bar{\sigma}_2^2$};
	\draw[color=red] (2,3) -- (4,4) node[midway,above] {\footnotesize $\bar{\sigma}_3^2$};
	
	\draw[color=blue] (0,0) -- (0.5,0.5) node[midway,right] {\footnotesize $\sigma_1^2$};
	\draw[color=blue] (0.5,0.5) -- (1,2) node[midway,right] {\footnotesize $\sigma_2^2$};
	\draw[color=blue] (1,2) -- (1.3,2.1) node[midway,below] {\footnotesize $\sigma_3^2$};
	\draw[color=blue] (1.3,2.1) -- (2,3) node[midway,below] {\footnotesize $\sigma_4^2$};
	\draw[color=blue] (2,3) -- (2.5,3.25) node[midway,below] {\footnotesize $\sigma_5^2$};
	\draw[color=blue] (2.5,3.25) -- (3.25,3.3) node[midway,below] {\footnotesize $\sigma_6^2$};
	\draw[color=blue] (3.25,3.3) -- (4,4) node[midway,below] {\footnotesize $\sigma_7^2$}; 
	
	\draw[color=black] (1,-2pt) node[below] {\footnotesize $\lambda^1$};
	\draw[color=black] (0.5,-3.5pt) node[below] {\footnotesize $\lambda_1$};
	\draw[color=black] (1,-15pt) node[below] {\footnotesize $\lambda_2$};
	\draw[color=black] (1.3,-3.5pt) node[below] {\footnotesize $\lambda_3$};
	\draw[color=black] (2,-2pt) node[below] {\footnotesize $\lambda^2$};
	\draw[color=black] (2,-15pt) node[below] {\footnotesize $\lambda_4$};
	\draw[color=black] (2.5,-3.5pt) node[below] {\footnotesize $\lambda_5$};
	\draw[color=black] (3.25,-3.5pt) node[below] {\footnotesize $\lambda_6$};
	\draw[color=black] (4,-2pt) node[below] {\footnotesize $1=\lambda^3$};
	\draw[color=black] (4,-15pt) node[below] {\footnotesize $\lambda_7$};
	
	\draw[dashed,color=gray] (0.5,0.5) -- (0.5,0);
	\draw[dashed,color=gray] (1,2) -- (1,0);
	\draw[dashed,color=gray] (1.3,2) -- (1.3,0);
	\draw[dashed,color=gray] (2,3) -- (2,0);
	\draw[dashed,color=gray] (2.5,3.25) -- (2.5,0);
	\draw[dashed,color=gray] (3.25,3.3) -- (3.25,0);
	\draw[dashed,color=gray] (4,4) -- (4,0);
	
	\end{tikzpicture}
	\caption{An example of variance and effective variance.}
\end{figure}
We denote the maximum by $\psi^{*}_N\coloneqq \max_{v\in V_N} \psi^N_v$. For any, possibly finite, sequence $\{ x_i \}_{i\geq 0}$ of real numbers we denote by $\Delta x_i= x_i-x_{i-1}$ the discrete increment. It turns out that the concave hull of $\mathcal{I}_{\sigma^2}$, denoted $\hat{\mathcal{I}}_{\sigma^2}$, gives the desired control for the first order of the maximum.
Arguin and Ouimet \cite[Theorem 1.2]{MR3541850} determined the correct first order behaviour, i.e. they showed that in probability,
\begin{equation}
		\lim_{N \rightarrow \infty} \frac{\psi^{*}_N}{2 \log(N)}= \mathcal{I}_{\bar{\sigma}}(1)=\sum_{i=1}^{m} \bar{\sigma}_i \Delta \lambda^i.
\end{equation}
In the following, the goal is to prove a second order correction and tightness of the maximum around its mean. Let $\pi_j$ be the unique index such that for $1 \leq j\leq m$ we have $\lambda^j= \lambda_{\pi_j}$.
Moreover, we write $t^j=\lambda^j \frac{\log N}{\log 2}$ as well as $t_j= \lambda_j \frac{\log N}{\log 2}$. We set
\begin{align}
m_N\coloneqq \sum_{j=1}^{m}2{\log2}\bar{\sigma}_j \Delta t^j  - \frac{(w_j \bar{\sigma}_j \log(\Delta t^j))}{4},
\end{align}
	where
\begin{equation}
w_j= \begin{cases}
3, \quad \mathcal{I}_{\bar{\sigma}^2}|_{( \lambda^{j-1},\lambda^j]} \equiv \mathcal{I}_{\sigma^2}|_{(\lambda^{j-1},\lambda^j]} \\
1, \quad else
\end{cases}
\end{equation}
The following theorem establishes tail estimates of the maximum centred by $m_N$.
\begin{thm}\label{thm1}
	Let $N\in \mathbbm{N}$ and $\{\psi^N_v \}_{v \in V_N}$ be a $2$d $(\sigma,\lambda)$-DGFF on $V_N$ with $M\in \mathbb{N}$ scales. Assume that on each interval $[\lambda^{i-1},\lambda^i]$ and $i=1,\dotsc,m$, we have either $\mathcal{I}_{\sigma^2} \equiv \mathcal{I}_{\bar{\sigma}^2} $ or $\mathcal{I}_{\sigma^2} < \mathcal{I}_{\bar{\sigma}^2}$
	There exist constants $C,c>0$ such that for any $x\in [0,\sqrt{\log N}]$,
	\begin{align}\label{eq:thm1_right_tail}
		C^{-1}\left(1+x\mathbbm{1}_{\sigma_1=\bar{\sigma}_1}\right)e^{-x \frac{2}{\bar{\sigma}_1}}\leq\mathbbm{P}\left(\max_{v \in V_N}\psi_v^N\geq m_N+x\right)\leq C(1+x\mathbbm{1}_{\sigma_1=\bar{\sigma}_1}) e^{-x \frac{2}{\bar{\sigma}_1}}.
	\end{align}
	and for any $0\leq \lambda \leq (\log\log N)^{\sfrac{2}{3}}$,
	\begin{align}\label{eq:thm1_left_tail}
		\mathbbm{P}\left(\max_{v \in V_N}\psi_v^N\leq m_N-\lambda\right)\leq Ce^{-c\lambda}.
	\end{align}
\end{thm}
Note that the result for the right-tail in \eqref{eq:thm1_right_tail} is precise up to a multiplicative constant.
For values $x>\sqrt{\log N}$, by Borell's inequality (see \autoref{theorem:borell}) and \cite[Lemma~A.3]{MR3541850}, there is a constant $c_{\sigma}\in (0, \infty)$, depending only on the variance parameter $\sigma$, such that
\begin{align}
\mathbbm{P}\left(|\psi_N^{*}-m_N|\geq x \right) \leq 2 e^{-c_{\sigma}x^2/\log(N)}.
\end{align}
As a simple consequence of \autoref{thm1}, we obtain the following corollary.
\begin{cor}\label{theorem:thm_tightness}
	Under the same assumptions of \autoref{thm1}, the sequence of the centred maximum $\{\psi_N^{*}-m_N\}_{N\geq 0}$ is tight.
	In particular,
	\begin{align}
	\mathbb{E}\left[\psi^*_N\right] =m_N+O(1),
	\end{align}
	where the term $O(1)$ is bounded by a constant which is uniform in $N$.
\end{cor}
An interesting fact is that the profile of the variance matters both for the leading term and the logarithmic correction. This phenomenon was first observed in the context of the GREM by Bovier and Kurkova \cite{Derrida2,MR2070334,MR2070335}, and in the context of the time-inhomogeneous branching Brownian motion/branching random walk by Bovier and Hartung \cite{MR3164771,MR3351476}, Fang and Zeitouni \cite{MR2968674}, Maillard and Zeitouni \cite{MR3531703} and Mallein \cite{MR3373310}.
\begin{rem}
	Regarding the additional assumption on the variance profile in \autoref{thm1}, we expect that in general there are essentially two properties which determine the logarithmic correction. For each interval $[\lambda^{j-1},\lambda^j]$ one has to see whether the effective variance and the real variance coincide in a neighbourhood at the beginning or the end of the interval. If neither is the case we have the $1/2$ correction. If it coincides in a neighbourhood at exactly one end point, we expect the factor to be $2/2$ and if it coincides in neighbourhoods at the beginning and the end, the correction factor should be $3/2$.
	If one considers the case of strictly decreasing variance $\sigma$ in \eqref{equation:1.5}, we expect the second order correction to be proportional to $\log^{1/3}(N)$ as observed in the analogue setting for variable-speed BBM \cite{MR2968674}.
\end{rem}
\subsection{Overview of related results}
In the case when $\sigma \equiv 1$, the 2d scale-inhomogeneous DGFF simply is the 2d DGFF. The maximum and more generally the extremal process of the DGFF has been the subject of intense investigations. Let $\phi_N^* \coloneqq \max_{v \in V_N} \phi_v^N$ be the maximum of the DGFF.
Through the works of Bolthausen, Deuschel and Giacomin \cite{MR1880237} as well as Bramson and Zeitouni \cite{MR2846636} one obtains,
\begin{align}
	\phi_N^*= 2 \log N - \frac{3}{4} \log \log N + Y,
\end{align}
where $Y$ is random variable of order $o(\log \log N)$ in probability. Bramson and Zeitouni further deduced that the centred maximum $\phi_N^*-\mathbb{E}\left[\phi_N^*\right]$ is tight as a sequence of real random variables.
Convergence of the centred maximum was then shown by Bramson, Ding and Zeitouni in \cite{MR3433630}. In \cite{MR3509015,2016arXiv160600510B}, Biskup and Louidor proved that the extremal process converges to a cluster Cox process.\\
Another closely related model is (variable-speed) branching Brownian motion (BBM).
It can be considered as the analogue model to the scale-inhomogeneous DGFF in the context of BBM. It first appeared in a paper by Derrida and Spohn \cite{MR971033}. To define variable-speed BBM, fix a Galton Watson tree, a time horizon $t>0$ and let $A:[0,1]\rightarrow [0,1]$, strictly increasing with $A(0)=0,\, A(1)=1$ and bounded second derivatives. The overlap $d(v,w)$ for leaves $v,w$ in the tree is the time of their most recent common ancestor. Variable-speed BBM in time $t$ and with time change $t A(\cdot / t)$ can then be defined as a centred Gaussian process $x$ indexed by the leaves of the tree and covariance $t A(d(v,w)/t)$, where $v$ and $w$ are leaves. BBM is the special case when $A(x)=x$ for $x\in [0,1]$, and coincides with the generalized random energy model (GREM) on the Galton-Watson tree. Compared to the 2d DGFF, its hierarchical structure makes it easier to analyse and the extremes of BBM are particularly well understood (see \cite{MR3101852,MR3129797,2014arXiv1412.5975B,MR0494541}). The extreme values and more general the extremal process for variable-speed BBM were investigated in \cite{MR3164771,MR3351476,MR2968674,MR2981635,MR3531703}. In particular, the first order and  second order correction of the maximum in the regime of weak correlations, i.e. when $A(s)<s$ for $s\in (0,1)$, is identical to the uncorrelated regime. In this regime, convergence of the extremal process was proved by Bovier and Hartung in \cite{MR3164771,MR3351476}. In the case of decreasing speed with finitely many changes in speed, the global maximum is a simple concatenation of the maximum at speed change. When the speed is strictly decreasing, i.e. when $A^{\prime \prime} <0$, Bovier and Kurkova \cite{MR2070334,MR2070335} showed that the first order is as in all other cases determined by the concave hull of $A$. The second order correction is no longer logarithmic but proportional to $t^{1/3}$, which was shown by Maillard and Zeitouni in \cite{MR3531703}, building upon the work by Fang and Zeitouni in \cite{MR2981635}.\\
In the discrete analogue model of (variable-speed) BBM, the (time-inhomogeneous) branching random walk (BRW) on the Galton Watson tree, there are results on the first and second order correction by Fang and Zeitouni \cite{MR2968674}, Mallein \cite{MR3361256} and Ouimet \cite{2015arXiv150908172O}. A notable difference in the context of (time-inhomogeneous) BRW is that one does not need to assume that increments are Gaussian (see \cite{MR3361256}). For the usual BRW, A\"{\i}d\'{e}kon proved convergence of the centred maximum \cite{MR3098680} and Madaule of the extremal process \cite{MR3615081}.
\subsection{Idea of proof}
The main idea to prove \autoref{thm1} is to use Gaussian comparison to compare the distribution of the centred maximum of the scale-inhomogeneous DGFF with the distribution of two auxiliary Gaussian fields, a time-inhomogeneous BRW (IBRW) and an modified inhomogeneous branching random walk (MIBRW). The time-inhomogeneous BRW is constructed in such a way that it is slightly less correlated than the scale-inhomogeneous DGFF which allows to use an available upper bound on the right tail of the maximum of the time-inhomogeneous BRW. The MIBRW has correlations that differ from those of the scale-inhomogeneous DGFF inside the field only up to a uniformly bounded constant. This allows, in a first step, to use Gaussian comparison to reduce the remaining lower bound on the right tail of the maximum to a corresponding lower bound on the right tail of the maximum of the MIBRW. In a second step, we prove the lower bound on the right tail of the centred maximum of the MIBRW that, together with the so-called ``sprinkling method'', also allows to deduce the upper bound on the left tail. The remaining lower bound on the right tail is achieved by a modified second moment analysis.\\

\textit{Outline of the paper:}
In the next section we  define two auxiliary Gaussian processes, the time-inhomogeneous branching random walk (IBRW) and the modified time-inhomogeneous branching random walk (MIBRW), and estimate their covariance structure.
In \autoref{sec:tail_estimates} we provide the necessary tail estimates that allow us to deduce \autoref{thm1}. We start with the upper bound on the right tail, then prove the lower bound on the right tail and finally, show the upper bound on the left tail.
In \autoref{appendix:A} we provide the Gaussian comparison theorems we use in the proof and Borell's Gaussian concentration inequality. In \autoref{appendix:B} we prove the covariance estimates stated in \autoref{sec:preparations}.

\section*{Acknowledgements}
I would like to thank my advisor, Anton Bovier, Lisa Hartung and an anonymous referee for their very careful reading and for their valuable comments that led to improvements of this paper.
\nocite{*}
%\subsection{Sketch of Proof}
%\input{sketch_proof.tex}
%\subsection{Structure of the paper}
%\input{structure.tex}

\section{Auxiliary processes and covariance estimates}\label{sec:preparations}

%\subsection{Sample path large deviation heuristics}
%\input{ldp.tex}
%\subsection{Auxiliary processes and approximations}
%\input{auxiliary_processes.tex}
%\subsection{Torus Gaussian free field}
%\input{tgff.tex}
%\subsection{Inhomogeneous branching random walks}
Consider $N=2^n$ for some $n \in \mathbbm{N}$. For $k = 0,1,\dotsc,n$ let $\mathcal{B}_k$ denote the collection of subsets of $\mathbbm{Z}^2$ consisting of squares of side length $2^k -1$ with corners in $\mathbbm{Z}^2$ and let $\mathcal{BD}_k$ denote the subset of $\mathcal{B}_k$ consisting of squares of the form $([0,2^k - 1]\cap \mathbbm{Z})^2+ (i2^k,j2^k).$ We remark that the collection $\mathcal{BD}_k$ partitions $\mathbbm{Z}^2$ into disjoint squares. For $v \in V_N,$ let $\mathcal{B}_k (v)$ denote those elements $B \in \mathcal{B}_k (v)$ with $v \in B$. Likewise define $\mathcal{BD}_k (v)$, i.e. for $v\in V_N$, $B\in \mathcal{BD}_k (v)$ if and only if $v\in B$. One should note that $\mathcal{BD}_k (x)$ contains exactly one element, whereas $\mathcal{B}_k (x)$ contains $2^{2k}$ elements.
\begin{definition}[Time-inhomogeneous branching random walk (IBRW)]\label{definition:ibrw}
	Let $\{a_{k,B}\}_{k\geq 0,B \in \mathcal{BD}_k}$ be an i.i.d. family of standard Gaussian random variables. We define the time-inhomogeneous branching random walk $ \{ R_z^N \}_{z\in V_N}$ by
\begin{equation}\label{equation:def_ibrw}
R_z^N (t)\coloneqq \sum_{k=n-t}^{n} \sum_{B \in \mathcal{BD}_k (z)} \sqrt{\log(2)} \tilde{\sigma}\left(\frac{n-k}{n}\right) a_{k,B},
\end{equation}
where $0 \leq t \leq n,$ $t \in \mathbbm{N}$ and $s\mapsto \tilde{\sigma}(s)$ is a non-negative function, for $s\in [0,1]$.
We specify the function $s\mapsto \tilde{\sigma}(s)$ later in the proof (see p. \pageref{eq:tilde_sigma}).
\end{definition}
It turns out that, due to its hierarchical structure, the IBRW is less correlated than the scale-inhomogeneous DGFF, which is beneficial to obtain upper bounds using Gaussian comparison. But this also makes it unsuitable to obtain sufficient lower bounds on the maximum value. We therefore introduce another auxiliary process whose covariance structure is much closer to the scale-inhomogeneous DGFF, and is defined by taking uniform averages of IBRWs.
For $v \in V_N$, let $\mathcal{B}_k^N (v)$ be the collection of subsets of $\mathbbm{Z}^2$ consisting of squares of size $2^k$ with lower left corner in $V_N$. For two sets $B,B^\prime\subset \mathbb{Z}^2$ we write $B\sim_N B^\prime$, if there exist integers $i,j$ such that $B^\prime= B+ (iN, jN)$. Let $\{b_{k,B} \}_{k \geq 0 , B \in \mathcal{B}_k^N}$ denote an i.i.d. family of centred  Gaussian random variables with unit variance and set
\begin{equation}
b_{k,B}^N \coloneqq \begin{cases}
b_{k,B}, \, B \in \mathcal{B}_k^N, \\
b_{k,B^{'}}, \, B \sim_N B^{'} \in \mathcal{B}_k^N.
\end{cases}
\end{equation}
\begin{definition}[Modified inhomogeneous branching random walk (MIBRW)]
	The modified inhomogeneous branching random walk (MIBRW) $\{S_v^N\}_{v \in V_N}$ is defined by
	\begin{equation}\label{equation:def_MIBRW}
	S_z^N (t) \coloneqq \sum_{k=n-t}^{n} \sum_{B \in \mathcal{B}_k^N (z)} 2^{-k}\sigma\left(\frac{n-k}{n}\right) b_{k,B}^N,
	\end{equation}
	where $0 \leq t \leq n,$ $t \in \mathbbm{N}$ and $\sigma$ is defined as in \eqref{parameters1}.
\end{definition}
%\subsection{Modified inhomogeneous branching random walks}
%\input{mbrw.tex}
\subsection{Covariance estimates}
In order to be able to apply Gaussian comparison, we need to compare the correlations of the processes introduced previously. We write $\log_{+}(x)=\max(0,\log_2 (x))$. Further, let $\|\cdot \|_2$ be the usual Euclidean distance and $\|\cdot \|_\infty$ the maximum distance. As we are working in two dimensions, they satisfy the relation $\|x-y\|_\infty \leq \|x-y\|_2 \leq \sqrt{2}\|x-y\|_\infty.$ In addition, we introduce for $v,w\in V_N$ two distances on the torus induced by $V_N$,
\begin{align}
	&d^N(v,w) \coloneqq \min_{z:\, z-w \in (N\mathbb{Z})^2} \|v-z\|_2, &d_\infty^N(v,w)\coloneqq \min_{z: \, z-w \in (N\mathbb{Z})^2}\|v-z\|_\infty.
\end{align}
Note that the Euclidean distance on the torus is smaller than the standard Euclidean distance, i.e. for all $v,w\in V_N$, it holds $d^N(v,w)\leq \|v-w\|_2.$
However, equality trivially holds if one restricts oneself on a smaller box, e.g. if $v,w \in (\sfrac{N}{4},\sfrac{N}{4}) + V_{\sfrac{N}{2}} \subset V_N.$
In the following we call $\{\tilde{S}^N_v\}_{v\in V_N}$ the homogeneous version of the process $\{S^N_v\}_{v\in V_N}$ which was introduced in \cite{MR2846636}, i.e. we assume that there is only one scale $\lambda_1=1$ with variance parameter $\sigma_1=1$.
\begin{lemma}\label{lemma:cov_comp}
	There exists a constant $C$ independent of $N=2^n$ such that for any $v,w \in V_N$,
	\begin{itemize}
		%\item $\left|R_{\tilde{Y}^N}(v,w)-\log(2) (n- \log_{+} (d^N(x,y)))\right| \leq C $ \label{cov_est_1}
		\item[i.] $\left|\mathbb{E}\left[\tilde{S}^N_v\tilde{S}^N_w\right]-(n- \log_{+}(d^N (x,y)))\right| \leq C, $
		%\item $\left|R_{Y^N}(v,w)-\log(2) \sum_{i=1}^{M} \sigma_i^2 [n\nabla \lambda_i \mathbbm{1}_{ n-\ceil*{\log_{+}(d^N(v,w))} \geq \lambda_i n}\right.\\\left. + ((1-\lambda_{i-1})n-\ceil*{\log_{+}(d^N(v,w))})\mathbbm{1}_{\lambda_{i-1}n<n-\ceil*{\log_{+}(d^N (v,w))}< \lambda_i n} ]\right| \leq C $
		\item[ii.] $\left|\mathbb{E}\left[S^N_v S^N_w\right] - n \mathcal{I}_{\sigma^2}\left(\frac{n-\log_{+} d^N(v,w)}{n}\right)\right| \leq C.$
	\end{itemize}
	Further, for any $x,y \in V_N +(2N,2N)\subset V_{4N},$
	\begin{itemize}
		\item[iii.] $\left|\mathbb{E}\left[\phi^{4N}_v \phi^{4N}_w \right] - \log(2)(n-\log_{+}(\|v-w\|_2) )\right| \leq C, $
		\item[iv.] $\left| \mathbb{E}\left[\psi^{4N}_v \psi^{4N}_w\right]-\log(2)\mathbb{E}\left[S^N_v S^N_w \right]\right| \leq C.$ \label{cov_est_6}
	\end{itemize}
	\begin{proof}
		See \autoref{appendix:B}.
	\end{proof}
\end{lemma}
\begin{rem}\label{rem:additional_assumptions}
	The assumption $N=2^n$ for $n\in \mathbbm{N}$ mainly simplifies notation and also the proof, however without removing essential difficulties.
\end{rem}
An important tool in the analysis of the scale-inhomogeneous DGFF is the Gibbs-Markov property of the DGFF. For two sets $U\subset V \subset \mathbb{Z}^2$ the DGFF on V can be decomposed into a sum of a DGFF on U and an independent Gaussian field, i.e.
\begin{align}
\phi^V_u \overset{d}{=} \phi^U_u + \mathbb{E}\left[\phi^V_u| \sigma\left(\phi^V_v:v\in V\setminus U^o \right)\right],\quad u\in V.
\end{align}
Further, if $A,B \subset V$ such that $A^o \cap B^o = \emptyset$, then $\{ \phi^V_u - \mathbbm{E}[\phi^V_u | \mathcal{F}_{\partial A}] \}_{u \in A}$ is a DGFF on A, independent of the DGFF on B $\{ \phi^V_u - \mathbbm{E}[\phi^V_u | \mathcal{F}_{\partial B}] \}_{u \in B}$.
%\subsubsection{Lower bound for the truncated MIBRW}
%\input{lower_bound_mibrw.tex}
\section{Tail estimates and tightness}\label{sec:tail_estimates}
The following analysis provides the necessary estimates to conclude \autoref{thm1}. 
\begin{comment}
	We use Slepian's Lemma to reduce the estimates on the maximum of the scale-inhomogeneous DGFF to corresponding estimates for the maximum of the MIBRW. We start with the estimates for the right tail. The upper bound is proven by comparing the maximum of the scale-inhomogeneous DGFF to the maximum of a suitable inhomogeneous branching random walk. To obtain the lower bound, we first use Slepian's lemma in order to reduce it to showing a lower bound for the maximum of the MIBRW. The corresponding tail estimates for the maximum of the MIBRW are obtained using a second moment computation. The last statement of \autoref{thm1} on its left tail follows by Gaussian comparison with the MIBRW and \autoref{lemma:up_bound_SN_left_tail}.
\end{comment}
\begin{lemma}\label{remark:assumption}
	There is a constant $\alpha_0>0$ such that for sufficiently large $N \in \mathbbm{N}$ and any $v,w \in V_N$, we have 
	\begin{align}\label{equation:bound_var_psi}
	\mathrm{Var}\left[\psi_v^N\right]\leq \log N \mathcal{I}_{\sigma^2}(1) + \alpha_0= \log N \sum_{i=1}^{M}\sigma_i^2 \Delta \lambda_{i} +\alpha_0,
	\end{align}
	and
	\begin{align}\label{equation:bound_cov_psi}
	\mathbbm{E}\left[(\psi_v^N-\psi_w^N)^2\right]\leq& 2 \log N\left[\mathcal{I}_{\sigma^2}(1)- \mathcal{I}_{\sigma^2}\left(\frac{n-\ceil*{ \log_{+}\|v-w\|_2}}{n}\right) \right] - \left|\mathrm{Var}\left[\psi_v^N\right]-\mathrm{Var}\left[\psi_w^N\right]\right| + 4 \alpha_0.
	\end{align}
	\begin{proof}
		Recall \autoref{definition:scale_inh_DGFF} and note that we have an underlying discrete Gaussian free field $\{\phi^N_v\}_{v\in V_N}$ such that $\psi^N_v=\sum_{i=1}^{M}\sigma_i \left(\phi^N_v(\lambda_i)-\phi^N_v(\lambda_{i-1}) \right)$, where $\phi^N_v(\lambda_i)-\phi^N_v(\lambda_{i-1})$ for $i=1,\dotsc, M$ are independent Gaussian free fields increments. A short computation shows that the variance of $\Delta \phi^N_v(\lambda_{i})$ is up to constants given by the difference of Green kernels on the boxes, that is $G_{[v]_{\lambda_{i}}}(v,v)-G_{[v]_{\lambda_{i-1}}}(v,v)$, for which we have a sufficient bound (see \cite[Lemma 3.10]{MR3526836}), and \eqref{equation:bound_var_psi} follows.\\
		For \eqref{equation:bound_cov_psi}, let $b_N(v,w)\coloneqq \max \left(\lambda\in [0,1]:\, [v]_\lambda \cap [w]_\lambda \neq \emptyset \right)$ be the branching scale for particles $v,w\in V_N$. For scales $\mu_i> \mu_i^\prime\geq b_N(v,w)$ and $i=1,2$, increments $\phi^N_v(\mu_1)-\phi^N_v(\mu_1^\prime)$ are independent of $\phi^N_w(\mu_2)-\phi^N_w(\mu_2^\prime)$. We define a set of representatives at scale $\lambda\in [0,1]$, denoted $R_\lambda$, such that it contains the centre of boxes that form a decomposition of $V_N$ into disjoint boxes with side length $N^{1-\lambda}$. Now, fix $v,w\in V_N$. There exists a set of representatives $R_\lambda$ at scale $\lambda= b_N(v,w) - \frac{4}{\log N}$, such that there is a common representative for $v$ and $w$, which we call $u_\lambda$.
		By \cite[Lemma A.6]{MR3541850}, there is a universal constant $C>0$ such that for $N$ large enough,
		\begin{align}\label{equation:5.5}
		\max_{u\in \{v,w\}} \mathbbm{E}\left[\left(\psi_u^N(\lambda)-\psi^N_{u_\lambda}(\lambda)\right)^2\right]\leq C,
		\end{align}
		We further note that increments of $v$ and $w$ beyond $b_N(v,w)$ are independent and that, by Cauchy-Schwarz,
		\begin{align}\label{equation:5.6}
			\mathbb{E}\left[\left(\psi^N_v(b_N(v,w))-\psi^N_v(\lambda)\right)\left(\psi^N_v(b_N(v,w)-\psi^N_v(\lambda))\right)\right]\leq \tilde{C}
		\end{align}
		as well as
		\begin{align}\label{equation:5.7}
			\max_{u\in \{v,w\}}\mathbb{E}\left[\left(\psi^N_u(b_N(v,w))-\psi^N_u(\lambda)\right)^2\right]\leq \tilde{C},
		\end{align}
		for some $\tilde{C}>0$. Thus, writing
		\begin{align}
			\psi^N_v-\psi^N_w=&\psi_v^N(\lambda)-\psi_{u_{\lambda}}^N(\lambda)+\psi_{u_{\lambda}}^N(\lambda)-\psi_w^N(\lambda)+ \psi^N_v(b_N)-\psi^N_v(\lambda)+\psi^N_w(b_N)-\psi^N_w(\lambda) +\psi_v^N-\psi_v^N(b_N)\nonumber\\
			&+\psi_w^N+\psi_w^N(b_N),
		\end{align}
		we can bound $\mathbb{E}\left[\left(\psi^N_v-\psi^N_w\right)^2\right]$ from above using \eqref{equation:5.5}, \eqref{equation:5.6}, \eqref{equation:5.7}, Green kernel estimates as for the first statement \eqref{equation:bound_var_psi}, as well as independence of increments beyond $b_N(v,w)$, which then implies the upper bound in \eqref{equation:bound_cov_psi}.
	\end{proof}
\end{lemma}
\begin{comment}
	\begin{lemma}\label{lemma:borell_applied_psi}
	Let $N\in \mathbbm{N}$ and $\{\psi_v^N\}_{v \in V_N}$ be a $(\sigma,\lambda)-$DGFF. For any $\delta>0$, there is a constant $c_{\sigma}(\delta)\in (0, \infty)$, depending only on the variance parameter $\sigma$ and the constant $\delta$, such that
	\begin{align}
	\mathbbm{P}\left(|\psi_N^{*}-\sqrt{\log(2)}M_n^{*}|\geq \delta \sqrt{\log(N)}\right) \leq e^{-c_{\sigma}(\delta)\log(N)}.
	\end{align}
	\begin{proof}
	Apply Borell's concentration inequality with \eqref{equation:bound_var_psi}.
	\end{proof}
	\end{lemma}
	This allows us to focus on the decay for deviations less than $O\left(\sqrt{\log(N)}\right)$.
\end{comment}

We begin with an upper bound on the right tail.
\begin{prop}\label{proposition:up_bound_right_tail}
There is a constant $C=C(\alpha_0)$, independent of $N$ such that for all $N \in \mathbbm{N}$ and $x>0$,
\begin{align}
	\mathbbm{P}\left(\max_{v \in V_N}\psi_v^N\geq m_N+x\right)\leq C(1+x\mathbbm{1}_{\sigma_1=\bar{\sigma}_1}) e^{-x \frac{2}{\bar{\sigma}_1}}.
\end{align}

\begin{comment}
	Furthermore for all $x\geq 1$, $y>0$, and $A \subset V_N$, we have:
	\begin{align}
	\mathbbm{P}\left(\max_{v \in A} \psi_v^N \geq \sqrt{\log(2)} M_N^{*}+z-y\right) \leq C(1+x)^{2d} \left(\frac{|A|}{|V_N|}\right)^{\sfrac{1}{2}}e^{-2(z-y)}.
	\end{align}
\end{comment}
\end{prop}
The principal idea to prove \autoref{proposition:up_bound_right_tail} is to use Gaussian comparison and compare the maximum of the scale-inhomogeneous DGFF to the maximum of suitable inhomogeneous branchig random walk. To obtain the correct upper bound we need to choose the variance of the IBRW appropriately.
Here, we need to distinguish two cases. If there exists exactly one effective variance parameter, then we choose $s\mapsto \tilde{\sigma}(s)$, such that $s\mapsto\mathcal{I}_{\tilde{\sigma}^2}(s)$ is the lower convex envelope of the function $s\mapsto \mathcal{I}_{\sigma^2}(s)$.
Else, if there are at least two effective scale parameter,s we introduce a parameter $0<\kappa\ll n$. We set $\sigma_{min}=\min_{1\leq i \leq M} \sigma_i$ and $\sigma_{max}=\max_{1\leq i \leq M} \sigma_i$. We pick $\tilde{\lambda}^1\equiv\tilde{\lambda}^1(\kappa)=\lambda^1 \frac{n}{n+k}$ as first effective scale and as first effective variance, $\bar{\sigma}_1$. Next, we set $\tilde{\lambda}_1=\tilde{\lambda}_1(\kappa)=\tilde{\lambda}^1 \frac{\sigma_{max}^2-\bar{\sigma}^1}{\sigma_{max}^2-\sigma_{min}^2}$, $\tilde{\lambda}_2=\tilde{\lambda}^1$ and $\tilde{\lambda}_3=\frac{\tilde{\lambda}^1\left(\bar{\sigma}_1- \sigma_{min}^2\right)+\left(\sigma_{max}^2-1\right)}{\sigma_{max}^2-\sigma_{min}^2}.$ For $s\in [0,1]$, we define the variance function as follows:
\begin{align}\label{eq:tilde_sigma}
	\tilde{\sigma}(s)=\left(\sigma_{min} \mathbbm{1}_{s\in [0,\tilde{\lambda}_1)} +\sigma_{max} \mathbbm{1}_{s\in [\tilde{\lambda}_1,\tilde{\lambda}_2)}\right)\mathbbm{1}_{\sigma_1\neq\bar{\sigma}_1} + \bar{\sigma}_1 \mathbbm{1}_{\sigma_1=\bar{\sigma}_1} +\sigma_{min} \mathbbm{1}_{s\in [\tilde{\lambda}_2,\tilde{\lambda}_3)} +\sigma_{max}^2 \mathbbm{1}_{s\in [\tilde{\lambda}_3,1]}.
\end{align}
In both cases our choice ensures that the first effective variances coincide, that $(n+\kappa)\mathcal{I}_{\tilde{\sigma}^2}\left(\frac{n-x}{n+\kappa}\right)\leq n\mathcal{I}_{\sigma^2}\left(\frac{n-x}{n}\right)$, for $x\in [0,n]$ and such that $\mathcal{I}_{\tilde{\sigma}^2}(1)=1$.
Before proving \autoref{proposition:up_bound_right_tail}, we need one more lemma.
\begin{lemma}\label{lemma:5.2}
	There is an integer $\kappa=\kappa(\alpha_0)>0$ such that for all $N \in \mathbbm{N}$, $\lambda \in \mathbbm{R}$ and $A \subset V_N$,
	\begin{align}
		\mathbbm{P}\left(\max_{v \in A}\psi^N_v\geq \lambda\right)\leq 2\mathbbm{P}\left(\max_{v \in 2^\kappa A}R_{v}^{2^{\kappa}N}\geq \lambda\right).
	\end{align}
\end{lemma}
\begin{proof}
	By \autoref{remark:assumption}, we can choose a sufficiently large constant $\kappa$ that depends only on $\alpha_0$, such that $\mathrm{Var}\left[\psi_v^N\right]\leq \log(2) \mathrm{Var}\left[R_{2^{\kappa}v}^{2^{\kappa}N}\right]$ for all $v \in V_N$. Thus,
	\begin{align}
		a_v^2\coloneqq \log(2)\mathrm{Var}\left[R_{2^{\kappa}v}^{2^{\kappa}N}\right] -\mathrm{Var}\left[\psi_v^N\right]
	\end{align}
	are non-negative. Let $X$ be a standard Gaussian. Since $\mathrm{Var}\left[R^N_v \right]=\mathrm{Var}\left[R^N_w\right],$ for all $v,w\in V_N$, we get
	\begin{align}
		\mathbbm{E}\left[(\psi_v^N+a_v X- \psi_w^N- a_w X)^2\right]&= \mathbbm{E}\left[(\psi_v^N-\psi_w^N)^2\right] + (a_v-a_w)^2 \nonumber \\ 
		& = \mathbbm{E}\left[(\psi_v^N-\psi_w^N)^2\right] + \left|\mathrm{Var}\left[\psi_v^N\right]-\mathrm{Var}\left[\psi_w^N\right]\right|\nonumber\\ 
		&\leq 2 \log (N)\left[1-  \mathcal{I}_{\sigma^2}\left(\frac{n-\ceil*{\log_{+}\|v-w\|_2}}{n}\right)\right] +4 \alpha_0, \label{equation:lb_8.10}
	\end{align}
	by \autoref{lemma:cov_comp}.
On the other hand by our choice of $\tilde{\sigma}$ in \eqref{eq:tilde_sigma}, $\mathrm{Var}\left[R^{2^{\kappa}N}_{2^{\kappa}v}\right]=\log(N)+\log (2)\kappa$ grows linearly in $\kappa$, whereas $\mathbbm{E}[R^{2^{\kappa}N}_{2^{\kappa}u}R^{2^{\kappa}N}_{2^{\kappa}v}]=\left(\log(N)+\log (2)\kappa\right)\mathcal{I}_{\tilde{\sigma}^2}\left(\frac{n-\log_{+} d^N(u,v)}{n+\kappa}\right)$. By our choice in \eqref{eq:tilde_sigma} and taking into account that for two vertices $u$ and $v$, $\log_{+}d^N(u,v)\geq \log_{+}\|u-v\|_2$,
	\begin{align}\label{equation:5.11}
		\mathbbm{E}\left[(R_{2^{\kappa}v}^{2^{\kappa}N}-R_{2^{\kappa}w}^{2^{\kappa}N})^2\right] &\geq 2 \left(\log(N)+\log (2)\kappa\right)\left[1-\mathcal{I}_{\tilde{\sigma}^2}\left(\frac{n-\ceil*{\log_{+}\|v-w\|_2}}{n+\kappa} \right)\right].
	\end{align}
	Combining \eqref{equation:5.11} with the upper bound in \eqref{equation:lb_8.10}, it follows that we may choose $\kappa(\alpha_0)$ such that for all $v,w \in V_N$,
	\begin{align}
		\mathbbm{E}\left[(\psi_v^N-\psi_w^N)^2\right]\leq \mathbbm{E}\left[(\psi_v^N+a_v X- \psi_w^N- a_w X)^2\right]\leq \mathbbm{E}\left[(R_{2^{\kappa}v}^{2^{\kappa}N}-R_{2^{\kappa}w}^{2^{\kappa}N})^2\right].
	\end{align}
	Applying Slepian's Lemma, we obtain for any $\lambda \in \mathbbm{R}_{+}$ and $A \subset V_N$,
	\begin{align}
		\mathbbm{P}\left(\max_{v \in A}\psi_v^N +a_v X\geq \lambda\right)\leq \mathbbm{P}\left(\max_{v \in 2^{\kappa}A}R_{v}^{2^{\kappa}N}\geq \lambda\right).
	\end{align}
	By independence and symmetry of $X$,
	\begin{align}
			\mathbbm{P}\left(\max_{v \in A}\psi_v^N\geq \lambda\right)\leq 2 \mathbbm{P}\left(\max_{v \in 2^{\kappa}A}R_{v}^{2^{\kappa}N}\geq \lambda\right).
	\end{align}
\end{proof}
\begin{proof}[Proof of \autoref{proposition:up_bound_right_tail}]
	\cite[Theorem~4.1]{MR3361256} gives us
	\begin{align}\label{eq:4.16}
	&\mathbbm{P}\left(\max_{v \in V_N}R_v^N\geq m_N+x\right)\leq C(1+x \mathbbm{1}_{\sigma_1=\bar{\sigma}_1})e^{-x \frac{2}{\bar{\sigma}_1}}, \quad\forall x\geq 0.
	\end{align}
	The claim follows from a combination of \autoref{lemma:5.2} and \eqref{eq:4.16}.
\end{proof}
Next, we prove a corresponding lower bound on the right tail.
\begin{lemma}\label{lemma:5.3}
	There is an integer $\kappa>0$ such that for all $N\in \mathbbm{N}$ and $\lambda \in \mathbbm{R}$,
	\begin{align}
		\frac{1}{2}\mathbbm{P}\left(\max_{v \in V_{2^{-\kappa}N}}\sqrt{\log(2)}S_v^{2^{-\kappa}N}\geq \lambda\right)\leq \mathbbm{P}\left(\max_{v \in V_N}\psi_v^N\geq \lambda\right).
	\end{align}
\end{lemma}
\begin{proof}
	Note that $(\frac{N}{4},\frac{N}{4})+2^{\kappa-3}V_{2^{-\kappa}N}\subset (\frac{N}{4},\frac{N}{4})+ V_{\frac{N}{8}} \subset V_N$. By \autoref{lemma:cov_comp} $ii.$ and $iv.$, there is a constant $C>0$, independent of $N$, such that
	\begin{align}
		&\left|\mathrm{Var}\left[\psi_{(\frac{N}{4},\frac{N}{4})+2^{\kappa-3}u}^N\right]-\mathrm{Var}\left[\psi_{(\frac{N}{4},\frac{N}{4})+2^{\kappa-3}v}^N\right]\right|\leq C, \quad \forall u,v\in V_{2^{-\kappa}N}.
	\end{align}
	Moreover, by $iv.$ in \autoref{lemma:cov_comp}
	\begin{align}
		&\mathrm{Var}\left[\psi_{(\frac{N}{4},\frac{N}{4})+2^{\kappa-3}v}^N\right]\geq \log(2)\mathrm{Var}\left[S_v^{2^{-\kappa}N}\right], \quad \forall v\in V_{2^{-\kappa}N},
	\end{align}
	for $\kappa>0$ large enough, independent of $N$. Thus, we can find a family of positive real numbers $\{a_v:v\in V_{2^{-\kappa}N}\}$ that satisfy $|a_u-a_v|\leq \sqrt{C}$ for a constant $C>0$, such that for $u,v\in V_N$ and an independent standard Gaussian random variable $X$,
	\begin{align}
		&\mathrm{Var}\left[\psi_{(\frac{N}{4},\frac{N}{4})+2^{\kappa-3}v}^N\right]=\log(2)\mathrm{Var}\left[S_v^{2^{-\kappa}N}+a_v X\right], \quad \forall v \in V_{2^{-\kappa}N}. \label{equation:8.22}
	\end{align}
	Using \autoref{lemma:cov_comp} $iv.$, and choosing $\kappa$ large enough, we have for $u,v\in V_{2^{-\kappa}N}$,
	\begin{align}\label{equation:8.19}
		\mathbbm{E}\left[(\psi_{(\frac{N}{4},\frac{N}{4})+2^{\kappa-3}u}^N-\psi_{(\frac{N}{4},\frac{N}{4})+2^{\kappa-3}v}^N)^2\right]\geq \log(2)\mathbbm{E}\left[(S_u^{2^{-\kappa}N}- S_v^{2^{-\kappa}N}+(a_u-a_v) X)^2\right].
	\end{align}
	Hence, by Slepian's Lemma we have for any $\lambda \in \mathbbm{R}$,
	\begin{align}
		\mathbbm{P}\left(\max_{v \in V_{2^{-\kappa}N}}\psi_{(\frac{N}{4},\frac{N}{4})+2^{\kappa-3}v}^N\geq \lambda\right)&\geq \mathbbm{P}\left(\sqrt{\log(2)}\max_{v \in V_{2^{-\kappa}N}}(S_v^{2^{-\kappa}N}+a_v X)\geq \lambda\right)\nonumber\\ &\geq \frac{1}{2} \mathbbm{P}\left(\sqrt{\log(2)}\max_{v \in V_{2^{-\kappa}N}}S_v^{2^{-\kappa}N} \geq \lambda\right),
	\end{align}
	as $X$ is an independent standard Gaussian.
\end{proof}
\begin{lemma}\label{lemma:right_tail_lb_MIBRW}
	Set $M_N^{*}\coloneqq m_N/\sqrt{\log(2)}$.
	There is a constant $C>0$ such that for any $N \in \mathbbm{N}$ and $y \in [0,\sqrt{\log N}]$,
	\begin{align}
		\mathbbm{P}\left(\max_{v \in V_N}S_v^N >M_N^{*} +y\right)\geq C\left(1+y\mathbbm{1}_{\sigma_1=\bar{\sigma}_1}\right)e^{- \frac{2\sqrt{\log(2)}}{\bar{\sigma}_1}y}.
	\end{align}
\end{lemma}
Recall the notation, i.e. $\pi_j$ is the unique index such that, for $1 \leq j\leq m,$ we have $\lambda^j= \lambda_{\pi_j}$ and
that we write $t^j=\lambda^j \frac{\log N}{\log 2}$ as well as $t_j= \lambda_j \frac{\log N}{\log 2}$. Moreover, we set
\begin{align}
M_N^{*}(t)\coloneqq \sum_{j=1}^{m} \frac{t \wedge t^j -t^{j-1}}{\Delta t^j}\left[2\sqrt{\log2}\bar{\sigma}_j \Delta t^j  - \frac{(w_j \bar{\sigma}_j \log(\Delta t^j))}{4\sqrt{\log(2)}}\right], \quad t\in \mathbb{R}_{+}.
\end{align}
The proof of \autoref{lemma:right_tail_lb_MIBRW} is based on a second moment computation. 
We introduce suitable events that control the paths that reach the maximum. For $v \in V_N^{'}= V_{\sfrac{N}{2}}+(\sfrac{N}{4},\sfrac{N}{4})\subset V_N, \,x \in \mathbbm{R},\, 0 \leq k \leq n$ and $ 0<i\leq m$, let
\begin{align}\label{equation:s_kn,x}
s_{k,n}(x) \coloneqq
\begin{cases}
\frac{\mathcal{I}_{\sigma^2}(\sfrac{k}{n})}{\mathcal{I}_{\sigma^2}(\lambda^1)}(x), &\text{ if } 0\leq k \leq \lambda^1,\\
\frac{\mathcal{I}_{\sigma^2}(\sfrac{k}{n},\lambda^i)}{\mathcal{I}_{\sigma^2}(\lambda^{i-1},\lambda^i)}(x), &\text{ if } \lambda^{i-1}<k\leq \lambda^{i}
\end{cases} 
\end{align}
be the 'optimal path' followed by extremal particles and
\begin{align}\label{equation:f_kn}
f_{k,n} \coloneqq 
\begin{cases}
C_f(\mathcal{I}_{\sigma^2}(\sfrac{k}{n})n)^{\sfrac{2}{3}}, &\text{ if } 0 \leq k \leq t_1, \\
C_f(\mathcal{I}_{\sigma^2}(\sfrac{k}{n}, \lambda^1)n)^{\sfrac{2}{3}}, &\text{ if } t_1<k \leq t^1,\\
C_f(\mathcal{I}_{\sigma^2}(\lambda^i,\sfrac{k}{n})n)^{\sfrac{2}{3}}, &\text{ if } t^{i}<k \leq t_{\pi_i + 1}: \, i \in \{1,\dotsc,m-1\}\\
C_f(\mathcal{I}_{\sigma^2}(\sfrac{k}{n}, \lambda^{i+1})n)^{\sfrac{2}{3}}, &\text{ if } t_{\pi_i +1}<k \leq t^{i+1}: \, i \in \{1,\dotsc,m-1\}
\end{cases}
\end{align}
be the concave barrier. The constant $C_f$ depends on the parameters and will be fixed later in the proof.
For $v \in V_N, \,x \in \mathbbm{R},\, \infty>y>0$ and $0 \leq k \leq n$, let
	\begin{align}
		I_n^y(1) \coloneqq [&\Delta M_N^{*}(t^1)+y-1,\Delta M_N^{*}(t^1)+y], \label{eq:4.28}\\
		I_n^y(i) \coloneqq [&\Delta M_N^{*}(t^i)-1,\Delta M_N^{*}(t^i)], \text{ for }1<i\leq m \label{eq:4.29} \\
		I_{k,n}(x) \coloneqq [&s_{k,n}(x)-f_{k,n},s_{k,n}(x)+f_{k,n}], \\
		C_v^{N,y}(r) \coloneqq \{&\Delta S^N_v(t^i) \in I_n^y(i), \, S^N_v (k+t^{i-1})-S_v^N(t^{i-1}) \in I_{k,n}(\Delta S_v^N(t^{i}))\nonumber\\ &\forall 0 < k < t^{i+1}-t^{i},\, 0< i\leq m: \, k+t^{i-1}\leq r \}, \\
		h_N(y) \coloneqq \text{ } &\sum_{v \in V^\prime_N} \mathbbm{1}_{C_v^{N,y}(t^m)}.
	\end{align}
$f_{k,n}$ and $s_{k,n}(x)$ are defined as before (see \eqref{equation:s_kn,x} and \eqref{equation:f_kn}). We can restrict the proof to the case of $m=1$ and to the assumption that $\mathcal{I}_{\sigma^2}(s)<\mathcal{I}_{\bar{\sigma}^2}(s)$ holds for all $0<s<1$. The statement in case of equality is given by \cite[Theorem 1.1]{MR3101848}. The lower bound then follows using the independence of increments and the fact that on the intervals, when $i>1$, we choose $y=0$, compare \eqref{eq:4.28} with \eqref{eq:4.29}. This implies that there is a constant $C>0$, such that we obtain as lower bound
\begin{align}
C\left(1+y \mathbbm{1}_{\sigma_1=\bar{\sigma}_1}\right)e^{- \frac{2\sqrt{\log(2)}}{\bar{\sigma}_1}y} \prod_{i=2}^{m} e^{- \frac{2\sqrt{\log(2)}}{\bar{\sigma}_i}0}\geq C\left(1+y \mathbbm{1}_{\sigma_1=\bar{\sigma}_1}\right) e^{- \frac{2\sqrt{\log(2)}}{\bar{\sigma}_1}y}.
\end{align}
Thus, until the end of the proof of \autoref{lemma:right_tail_lb_MIBRW}, we restrict ourselves to the case when $m=1$ and $\mathcal{I}_{\sigma^2}(s)<\mathcal{I}_{\bar{\sigma}^2}(s)$ holds for all $0<s<1$.
\begin{lemma}\label{lemma:lb_1st_moment_hny}
	There are constants $C,c>0$ such that it holds for all $N\in \mathbbm{N}$ sufficiently large and $y\in [0,\sqrt{\log N}]$,
	\begin{align}\label{eq:4.33}
		c\geq\mathbbm{E}\left[h_N(y)\right]\geq C e^{-\frac{2\sqrt{\log(2)}}{\bar{\sigma}_1}y}.
	\end{align}
\end{lemma}
\begin{lemma}\label{lemma:up_2nd_moment_hny}
	There is a constant $\tilde{C}>0$ independent of $N$, such that, for $y \in [0,\sqrt{\log N}]$,
	\begin{align}
		\mathbbm{E}\left[h_N^2(y)\right]\leq \mathbbm{E}\left[h_N(y)\right]^2+ (1+\tilde{C})\mathbbm{E}\left[h_N(y)\right].
	\end{align}
\end{lemma}
\begin{proof}[Proof of \autoref{lemma:lb_1st_moment_hny}]
	In the following, we write $M_N^{*}$ instead of $M_N^{*}(t^1)$.
	By linearity of expectations,
	\begin{align}\label{equation:4.53}
	\mathbbm{E}\left[h_N\right]&=\frac{1}{4}2^{2t^1}\mathbbm{P}(S_v^N(t^1)\in I_n(1), \, S_v^N(k) \in I_{k,n}(S_v^N(t^1))\text{ for } 0<k<t^1).
	\end{align}
	Note that $\mathbb{E}\left[s_{k,n}(S^N_v(t^1))\left(S^N_v(k)-s_{k,n}(S^N_v(t^1))\right) \right]=0$, and so 
	\begin{align}\label{equation:4.54}
	\mathrm{Var}\left[S_v^N(k)-s_{k,n}(S_v^N(t^1))\right]&= \mathrm{Var}\left[S_v^N(k)-s_{k,n}(S_v^N(t^1))\right]  =n \mathcal{I}_{\sigma^2}\left(\frac{k}{n}\right)\left(1-\frac{\mathcal{I}_{\sigma^2}(\sfrac{k}{n})}{\mathcal{I}_{\sigma^2}(\lambda^1)}\right).
	\end{align}
	In particular, $\mathbbm{E}\left[S_v^N(k)-s_{k,n}(S_v^N(t^1))\right]=0$. Under our assumptions, we have .
	By conditioning the last event in \eqref{equation:4.53} on $S^N_v(t^1)$, using that this is independent of $\{S^N_v(k)-s_{k,n}(S^N_v(t^1))\}_{k=0}^{t^1}$, we have
	\begin{align}\label{equation:lem5.5_eq1}
		\mathbb{E}\left[h_N(y)\right]= \frac{1}{4}2^{2n} \mathbbm{P}\left(S_v^N (t^1)\in [M_N^{*}+y-1,M_N^{*}+y]\right) \mathbbm{P}\left(S_v^N(k)\in I_{k,n}(S_v^N(t^1), \, 0<k<t^1\right).
	\end{align}
	To estimate the first probability in \eqref{equation:lem5.5_eq1}, note that $S^N_v(t^1)\sim \mathcal{N}\left(0,\bar{\sigma}_1^2 t^1\right)$ and that the assumptions imply the identity $M_N^{*}=2\sqrt{\log(2)}\bar{\sigma}_1 n - \frac{1}{4 \sqrt{\log(2)}} \log(n) \bar{\sigma}_1$. Thus, by a standard Gaussian estimate,
	\begin{align}\label{equation:5.30}
	\mathbbm{P}\left(S_v^N(t^1)\in I_n^y(1)\right)=\int_{M_N^{*}+y-1}^{M_N^{*}+y} \frac{\exp\left[-x^2/(2 \bar{\sigma}_1^2 t^1)\right]}{\sqrt{2 \pi \bar{\sigma}_1^2 t^1}} \mathrm{d}x\geq  \frac{\exp\left[-(M_N^{*}+y)^2/(2\bar{\sigma}_1^2 t^1)\right]}{\sqrt{2\pi \bar{\sigma}_1^2 t^1}}.
	\end{align}
	By expanding the square in \eqref{equation:5.30} and bounding all terms in the exponential that tend to $0$ as $n\rightarrow \infty$ by a constant, we can find a constant $C>0$ such that
	\begin{align}\label{eq:4.40}
	\mathbbm{P}\left(S_v^N(t^1)\in I_n^y(1)\right)\geq C N^{-2} e^{-y \frac{2\sqrt{\log(2)}}{\bar{\sigma}_1}}.
	\end{align}
	We turn to the second probability in \eqref{equation:lem5.5_eq1}.
	By subadditivity of measures and using \eqref{equation:4.54},
	\begin{align}\label{equation:4.70}
	\mathbbm{P}(S_v^N(k)\in I_{k,n}(S_v^N(t^1)), \, 0<k<t^1) &\geq 1 - 2 \sum_{k=1}^{t^1-1}\mathbbm{P}(S_v^N(k)-s_{k,n}(S_v^N(t^1))>f_{k,n})\nonumber \\ &\geq 1- 2\sum_{k=1}^{t^1-1}C\exp\left[-\frac{1}{2}\frac{f_{k,n}^2}{\mathcal{I}_{\sigma^2}(\sfrac{k}{n})n(1-\frac{\mathcal{I}_{\sigma^2}(\sfrac{k}{n})}{\mathcal{I}_{\sigma^2}(\lambda^1)})}\right].
	\end{align}
	By definition of the concave barrier in \eqref{equation:f_kn}, we may split and bound the sum in \eqref{equation:4.70} from above by
	\begin{align}\label{eq:4.41}
	&  \sum_{k=1}^{t_1} C\exp\left[-\frac{1}{2}C_f^2 \sigma_1^{\sfrac{2}{3}}k^{\sfrac{1}{3}}\right]\mathbbm{1}_{\sigma_1 \neq 0} + \sum_{k=t_1 +1}^{t^1 -1} C\exp\left[-\frac{1}{2}C_f^2 \min_{i \in \{2, \dotsc,\pi_1\}:\sigma_i>0}(\sigma_i)^{\sfrac{2}{3}}(t^1-k)^{\sfrac{1}{3}}\right]
	< \frac{c}{2}, 
	\end{align}
	where  $0<c<1$ is a constant independent of $n$, if $C_f$ is large enough.
	Inserting \eqref{eq:4.41} into \eqref{equation:4.70} gives
	\begin{align}\label{equation:4.71}
	\mathbbm{P}(S_v^N(k)\in I_{k,n}(S_v^N(t^1)), \, \forall 0<k<t^1) > 1-c=c_2>0.
	\end{align}
	Inserting \eqref{equation:4.71} and \eqref{eq:4.40} into \eqref{equation:lem5.5_eq1} finishes the proof of the lower bound in \eqref{eq:4.33}.
	To get an upper bound in \eqref{equation:lem5.5_eq1} we bound the second probability by $1$. For the first probability, as for the lower bound, we get
	\begin{align}
		\mathbbm{P}\left(S_v^N(t^1)\in I_n^y(1)\right)&\leq C N^{-2} \exp\left[-(y-1) \frac{2\sqrt{\log(2)}}{\bar{\sigma}_1} \right].
	\end{align}
	Inserting this into \eqref{equation:lem5.5_eq1}, we obtain the upper bound in \eqref{eq:4.33}.
\end{proof}
\begin{proof}[Proof of \autoref{lemma:up_2nd_moment_hny}]
	As in the proof of \autoref{lemma:lb_1st_moment_hny}, we write $M_N^{*}$ instead of $M_N^{*}(t^1)$.
	Recall that, for $v,w \in V_N$, $r(v,w)= n - \lceil \log_2(d_N^\infty(v,w)+1)\rceil$ denotes the number of scales of independent increments of the processes $S_v^N(k)$ and $S_w^N(k^\prime)$. By decomposing the second moment along $r(\cdot,\cdot)$ and using independence of the increments,
	\begin{align}\label{equation:5.37}
		\mathbbm{E}\left[h_N^2(y)\right]&= \sum_{v,w \in V^\prime_N} \mathbbm{P}\left(C_v^{N,y}(t^1) \cap C_w^{N,y}(t^1)\right)
		= \sum_{k=0}^{n}\sum_{\substack{v,w \in V^\prime_N \\ r(v,w)=k}} \mathbbm{P}\left(C_v^{N,y}(t^1) \cap C_w^{N,y}(t^1)\right)\nonumber\\
		&\leq \mathbbm{E}\left[h_N(y)\right]^2+ \mathbbm{E}\left[h_N(y)\right]
		+ \sum_{k=1}^{n-1}\sum_{\substack{v,w \in V^\prime_N \\ r(v,w)=k}}\mathbbm{P}\left(C_v^{N,y}(t^1) \cap C_w^{N,y}(t^1) \right).
	\end{align}
	To bound the double sum from above, we bound each summand from above.
	Fix $v,w \in V_N^{'}$ with $r(v,w)=r=k \in \{1,\dotsc,n-1\}$. We set $B_{k,n}(x)\coloneqq [x-s_{k,n}(x)-f_{r,n},x-s_{k,n}(x)+f_{r,n}]$. Dropping the constraint for $w$ up to time $r$, we have
	\begin{align}
	\mathbbm{P}\left(C_v^{N,y} (t^1)\cap C_w^{N,y}(t^1) \right)\leq &\mathbbm{P}\left(C_v^{N,y}(t^1) \cap C_w^{N,y}(r)\right)\max_{x\in I_n(1)} \mathbbm{P}\left(S_w^N(t^1)-S_w^N(r)\in B_{r,n}(x)\right)\nonumber\\ \leq &\mathbbm{P}\left(C_v^{N,y}(t^m)\right) \max_{x \in I_n(1)}\mathbbm{P}\left(S_w^N(t^1)-S_w^N(r)\in B_{r,n}(x)\right).
	\end{align}
	For fixed $v \in V_N^{'}$, the number of points $w\in V_N^{'}$ that satisfy $d_N^{\infty}(v,w)\in [2^k,2^{k+1}],$ is bounded by $c_1 2^{2k}=2^{2(t^1-r)}$ for some $c_1>0$. Therefore, we can bound the last summand in \eqref{equation:5.37} from above by
	\begin{align}\label{equation:5.39}
		c_1 \mathbbm{E}\left[h_N(y)\right] \sum_{r=1}^{n-1}2^{2(t^1-r)} \max_{\substack{x \in I_n^y(1)\\v \in V_N}} \mathbbm{P}\left(S_v^N(t^1)- S_v^N(r) \in x + I_{r,n}(x)\right).
	\end{align}
	To bound the probability in \eqref{equation:5.39}, we use that for any $x \in I_n^y(1)$,
	\begin{align}
		 A_{r,n,x}^y &\coloneqq \mathbbm{P}\left(S_v^N(t^1)-S_v^N(r)\in x +I_{r,n}(x)\right) =\int_{x-s_{r,n}(x)-f_{r,n}}^{x-s_{r,n}(x)+f_{r,n}} \frac{\exp\left[-\frac{1}{2} \frac{z^2}{\mathcal{I}_{\sigma^2}(\sfrac{r}{n},\lambda^1)n}\right]}{\sqrt{2 \pi \mathcal{I}_{\sigma^2}(\sfrac{r}{n},\lambda^1)n}} \mathrm{d}z\nonumber \\
		&\leq \frac{2 f_{r,n}}{\sqrt{\mathcal{I}_{\sigma^2}(\sfrac{r}{n},\lambda^1)n}}\exp\left[-\frac{1}{2}\frac{(M_N^{*}+y -s_{r,n}(M_N^{*}+y)-f_{r,n} )^2}{\mathcal{I}_{\sigma^2}(\sfrac{r}{n},\lambda^1)n}\right]. \label{equation:arnxy_exp}
	\end{align}
	Noting that $n=t^1$ and using \eqref{equation:s_kn,x}, we bound from below the square in the exponential in \eqref{equation:arnxy_exp} by
	\begin{align}
		(M_N^{*}+y)^2\left(1-\frac{\mathcal{I}_{\sigma^2}(\sfrac{r}{n})}{\mathcal{I}_{\sigma^2}(\lambda^1)}\right)^2 - 2f_{r,n}\frac{\mathcal{I}_{\sigma^2}(\sfrac{r}{n},\lambda^1)}{\mathcal{I}_{\sigma^2}(\lambda^1)}(M_N^{*}+y)= (M_N^{*})^2\left(1-\frac{\mathcal{I}_{\sigma^2}(\sfrac{r}{n})}{\mathcal{I}_{\sigma^2}(\lambda^1)}\right)^2-2f_{r,n}M_N^{*}\frac{\mathcal{I}_{\sigma^2}(\sfrac{r}{n},\lambda^1)}{\mathcal{I}_{\sigma^2}(\lambda^1)}\nonumber\\+(2M_N^{*}y+y^2)\left(1-\frac{\mathcal{I}_{\sigma^2}(\sfrac{r}{n})}{\mathcal{I}_{\sigma^2}(\lambda^1)}\right)^2-2 yf_{r,n}\frac{\mathcal{I}_{\sigma^2}(\sfrac{r}{n},\lambda^1)}{\mathcal{I}_{\sigma^2}(\lambda^1)} . \label{eq:lb 8.51}
	\end{align}
	Inserting \eqref{eq:lb 8.51} into \eqref{equation:arnxy_exp}, dropping the term involving $y^2$, and noting that we can bound the term $\exp\left[\frac{2 yf_{r,n}\frac{\mathcal{I}_{\sigma^2}(\sfrac{r}{n},\lambda^1)}{\mathcal{I}_{\sigma^2}(\lambda^1)}}{\mathcal{I}_{\sigma^2}(\sfrac{r}{n},\lambda^1)n}\right]$ by a constant, we obtain that \eqref{equation:arnxy_exp} is bounded from above by
	\begin{align}\label{eq:4.50}
		\frac{C f_{r,n}}{\sqrt{\mathcal{I}_{\sigma^2}(\sfrac{r}{n},\lambda^1)n}} &\exp \left[-\frac{1}{2} \frac{(M_N^{*})^2\left(1-\frac{\mathcal{I}_{\sigma^2}(\sfrac{r}{n})}{\mathcal{I}_{\sigma^2}(\lambda^1)}\right)^2-2f_{r,n}M_N^{*}\frac{\mathcal{I}_{\sigma^2}(\sfrac{r}{n},\lambda^1)}{\mathcal{I}_{\sigma^2}(\lambda^1)}+2yM_N^{*}\left(1-\frac{\mathcal{I}_{\sigma^2}(\sfrac{r}{n})}{\mathcal{I}_{\sigma^2}(\lambda^1)}\right)^2}{\mathcal{I}_{\sigma^2}(\sfrac{r}{n},\lambda^1)n} \right]\nonumber\\
		&\leq \frac{ C f_{r,n}}{\sqrt{\mathcal{I}_{\sigma^2}(\sfrac{r}{n},\lambda^1)n}}\exp\left[-2\log(2)t^1 \frac{\mathcal{I}_{\sigma^2}(\sfrac{r}{n},\lambda^1)}{\mathcal{I}_{\sigma^2}(\lambda^1)} + \frac{\left(1+\frac{y}{4\sqrt{\log(2)}t^1}\right)}{2}\log(t^1)\frac{\mathcal{I}_{\sigma^2}(\sfrac{r}{n},\lambda^1)}{\mathcal{I}_{\sigma^2}(\lambda^1)} \right. \nonumber
		\\ &\qquad \quad    \left.-2y\sqrt{\log(2)}\frac{\mathcal{I}_{\sigma^2}(\sfrac{r}{n},\lambda^1)}{\mathcal{I}_{\sigma^2}(\lambda^1)}-\frac{\log(t^1)^2}{32 \log(2)t^1}\frac{\mathcal{I}_{\sigma^2}(\sfrac{r}{n},\lambda^1)}{\mathcal{I}_{\sigma^2}(\lambda^1)}+\frac{C_f \sigma_1^{\sfrac{4}{3}}r^{\sfrac{2}{3}}}{(4\log(2))^{-\sfrac{1}{2}}\bar{\sigma}_1}\right].
	\end{align}
	Let $i$ be minimal such that $\sigma_i>0$.
	We distinguish the cases $0<r\leq t_i$ and $t_i<k<t^1$.
	We may assume that $\sigma_1>0$.\\
	\underline{Case 1:} In this case, we have $\mathrm{Var}[S_v^N(t^1)-S_v^N(r)]= \mathcal{I}_{\sigma^2}(\sfrac{r}{n},\lambda^1)n$ and $f_{r,n}=C_f (\sigma_1^2r)^{\sfrac{2}{3}}$. Since $r\leq t_1$, $\frac{\frac{1}{\lambda_1}\mathcal{I}_{\sigma^2}(\lambda_1)}{\frac{1}{\lambda^1}\mathcal{I}_{\sigma^2}(\lambda^1)}=\frac{\sigma_1^2}{\bar{\sigma}_1^2}\in (0,1),$ and so there is an $\eta_1<1$, independent of $r$ and $n$, such that
	 \begin{align}
	 \frac{\mathcal{I}_{\sigma^2}(\sfrac{r}{n},\lambda^1)}{\mathcal{I}_{\sigma^2}(\lambda^1)}t^1= t^1-t^1 \frac{\mathcal{I}_{\sigma^2}(\sfrac{r}{n})}{\mathcal{I}_{\sigma^2}(\lambda^1)}= t^1-r\frac{\frac{1}{\sfrac{r}{n}}\mathcal{I}_{\sigma^2}(\sfrac{r}{n})}{\frac{1}{\lambda^1}\mathcal{I}_{\sigma^2}(\lambda^1)}=t^1-r\frac{\frac{1}{\lambda_1}\mathcal{I}_{\sigma^2}(\lambda_1)}{\frac{1}{\lambda^1}\mathcal{I}_{\sigma^2}(\lambda^1)}=t^1-\eta_1 r.
	 \end{align}
	 Similarly, we have $\frac{\mathcal{I}_{\sigma^2}(\sfrac{r}{n},\lambda^1)}{\mathcal{I}_{\sigma^2}(\lambda^1)} \geq 1- \frac{\sigma_1^2}{\bar{\sigma}_1^2} \frac{\lambda_1}{\lambda^1}$. Using these facts in \eqref{eq:4.50}, we get
	 \begin{align}
	 	A^y_{r,n,x}&\leq C r^{\sfrac{2}{3}} \exp\left(\tilde{C}r^{\sfrac{2}{3}}\right) 2^{-2(t^1-\eta_1 r)} \frac{\exp\left[\log(t^1)\frac{\mathcal{I}_{\sigma^2}(\sfrac{r}{n},\lambda^1)}{2\mathcal{I}_{\sigma^2}(\lambda^1)}-2y\sqrt{\log(2)}\left(1-\frac{\sigma_1^2}{\bar{\sigma}_1^2}\frac{\lambda_1}{\lambda^1}\right)\right]}{\sqrt{\mathcal{I}_{\sigma^2} (\sfrac{r}{n},\lambda^1)t^1}}\nonumber\\
	 	&\qquad\times \exp\left[-\log(t^1)\frac{\mathcal{I}_{\sigma^2}(\sfrac{r}{n},\lambda^1)}{2t^1\mathcal{I}_{\sigma^2}(\lambda^1)} \left(\frac{\log(t^1)-4\sqrt{\log(2)}y}{16 \log(2)}\right)\right].
	 \end{align}
	 Note that we have $\frac{\mathcal{I}_{\sigma^2}(\sfrac{r}{n},\lambda^1)}{\mathcal{I}_{\sigma^2}(\lambda^1)}<1$ and $\left(\frac{\log(t^1)-4\sqrt{\log(2)}y}{16 \log(2)}\right)\geq 0$. Thus, in the case $0<r\leq t_1$, we have
	 \begin{align}\label{equation:5.451}
	 	A_{r,n,x}^y\leq& C r^{\sfrac{2}{3}} 2^{-2(t^1-\eta_1 r)}
	 	  \exp\left[\tilde{C}r^{\sfrac{2}{3}}-\frac{2\sqrt{\log(2)}}{\bar{\sigma}_1}\left(1-\frac{\sigma_1^2}{\bar{\sigma}_1^2}\frac{\lambda_1}{\lambda^1}\right)y\right]\nonumber\\
	 	  \leq& C2^{-2(t^1-\eta_1 r)+o(r)} \exp\left[-\frac{2\sqrt{\log(2)}}{\bar{\sigma}_1}\left(1-\frac{\sigma_1^2}{\bar{\sigma}_1^2}\frac{\lambda_1}{\lambda^1}\right)y\right].
	 \end{align}
	 Note that for the last factor in the exponent we know $0<1-\frac{\sigma_1^2}{\bar{\sigma}_1^2}\frac{\lambda_1}{\lambda^1}<1$, which guarantees that we have the correct sign to have sufficient decay in $y$.
	 \\
	 \underline{Case 2:}
	 The same computation as in \eqref{equation:arnxy_exp}, now in the case of $t_1<r<t^1,$ $f_{r,n}= C_f(\mathcal{I}_{\sigma^2}(\sfrac{r}{n},\lambda^1))^{\sfrac{2}{3}}$ and $x\in I^y_n(1)$, yields
	 \begin{align}\label{eq:4.56}
	 	A^y_{r,n,x}&\leq  \frac{2 f_{r,n}}{\sqrt{\mathcal{I}_{\sigma^2} (\sfrac{r}{n},\lambda^1)n}}\exp\left[-\frac{1}{2}\frac{(M_N^{*}(t^1)+y-s_{r,n}(M_N^{*}(t^1)+y)-f_{r,n})^2}{\mathcal{I}_{\sigma^2}(\sfrac{r}{n},\lambda^1)t^1}\right]\nonumber\\ &\leq C 2^{-2t^1\frac{\mathcal{I}_{\sigma^2}(\sfrac{r}{n},\lambda^1)}{\mathcal{I}_{\sigma^2}(\lambda^1)}} (\mathcal{I}_{\sigma^2}(\sfrac{r}{n},\lambda^1)n)^{\sfrac{1}{6}}\exp\left[\log(t^1)\frac{\mathcal{I}_{\sigma^2}(\sfrac{r}{n},\lambda^1)}{2\mathcal{I}_{\sigma^2}(\lambda^1)}+\frac{C_f(\mathcal{I}_{\sigma^2}(\sfrac{r}{n},\lambda^1)n)^{\sfrac{2}{3}}}{(4\log(2))^{-\sfrac{1}{2}}\bar{\sigma}_1}\right]
	 	\nonumber\\
	 	&\qquad \times \exp\left[-y\frac{2\sqrt{\log(2)}}{\bar{\sigma}_1}\frac{\mathcal{I}_{\sigma^2}(\sfrac{r}{n},\lambda^1)}{\mathcal{I}_{\sigma^2}(\lambda^1)}-\log(t^1)\frac{\mathcal{I}_{\sigma^2}(\sfrac{r}{n},\lambda^1)}{2t^1\mathcal{I}_{\sigma^2}(\lambda^1)} \left(\frac{\log(t^1)-4\sqrt{\log(2)}y}{16 \log(2)}\right)\right].
	 \end{align}
	 As $y\in [0,\sqrt{\log N}]$, $\left(\frac{\log(t^1)-4\sqrt{\log(2)}y}{16 \log(2)}\right)\geq 0$. Moreover, for $t_1<r <t^1$,
	 \begin{align}
	 t^1\frac{\mathcal{I}_{\sigma^2}(\sfrac{r}{n},\lambda^1)}{\mathcal{I}_{\sigma^2}(\lambda^1)}= \frac{  \frac{1}{\lambda^1-\sfrac{r}{n}}\mathcal{I}_{\sigma^2}(\sfrac{r}{n},\lambda^1)}{\frac{1}{\lambda^1}\mathcal{I}_{\sigma^2}(\lambda^1)}(t^1-r)\geq \eta_2(t^1-r),
	 \end{align}
	 for a constant $\eta_2>1$ that is independent of $r$ and $n$. Using these facts in \eqref{eq:4.56},
	 we obtain
	 \begin{align}\label{equation:5.46}
	 	A_{r,n,x}^y\leq& C 2^{-\eta_2(t^1-r)} (\mathcal{I}_{\sigma^2}(\sfrac{r}{n},\lambda^1)n)^{\sfrac{2}{3}} \exp\left[C_f(\mathcal{I}_{\sigma^2}(\sfrac{r}{n},\lambda^1)n)^{\sfrac{2}{3}}2\sqrt{\log(2)}-y\left(\frac{2 \sqrt{\log(2)}}{\bar{\sigma}_1}\frac{\mathcal{I}_{\sigma^2}(\sfrac{r}{n},\lambda^1)}{\mathcal{I}_{\sigma^2}(\lambda^1)}\right)\right]\nonumber\\
	 	\leq & C 2^{-2 \eta_2 (t^1-r)+o(t^1-r)}\exp\left[-y\left(\frac{2 \sqrt{\log(2)}}{\bar{\sigma}_1}\frac{\mathcal{I}_{\sigma^2}(\sfrac{r}{n},\lambda^1)}{\mathcal{I}_{\sigma^2}(\lambda^1)}\right)\right] .
	 \end{align}
	 Combining the bounds in \eqref{equation:5.451} and \eqref{equation:5.46} and observing that both $(1-\eta_1)>0$ and $(1-\eta_2)<0$ hold and using $y\geq 0$, allows us to bound the sum in \eqref{equation:5.39} by an absolute constant $C_2>0$, i.e.
	 \begin{align}\label{eq:4.57}
	 	\sum_{r=1}^{n-1} 2^{2(t^1-r)} \max_{x \in I^y_n(1)}A^y_{r,n,x} \leq  C\left[\sum_{r=1}^{t_1}2^{-2r(1-\eta_1)+o(r)}+\sum_{r=t_1 + 1}^{t^1-1}2^{2(1-\eta_2)(t^1-r)+o(t^1-r)}\right] \leq C_2.
	 \end{align}
	 Inserting \eqref{eq:4.57} into \eqref{equation:5.37} concludes the proof.
\end{proof}
\begin{proof}[Proof of \autoref{lemma:right_tail_lb_MIBRW}]
	Combining \autoref{lemma:lb_1st_moment_hny} with \autoref{lemma:up_2nd_moment_hny} shows that there are constants, $\tilde{C}, C,c>0,$ such that
	\begin{align}
		\mathbbm{P}\left(\max_{v \in V_N} S_v^N > M_N^{*}+y\right)&\geq \mathbbm{P}(h_N(y)\geq 1) \geq \frac{\left(\mathbbm{E}\left[h_N(y)\right]\right)^2}{\mathbbm{E}\left[h_N^2(y)\right]}\geq \frac{\mathbbm{E}\left[h_N(y)\right]^2}{\mathbbm{E}\left[h_N(y)\right]^2+(1+\tilde{C})\mathbbm{E}\left[h_N(y)\right])} \nonumber
		\\ &\geq \frac{\mathbbm{E}\left[h_N(y)\right]}{1+c}\geq C e^{-y \frac{2\sqrt{\log(2)}}{\bar{\sigma}_1}} .
	\end{align}
\end{proof}
The goal in the following is to provide an upper bound on the left tail of the centred maximum of the $(\sigma,\lambda)-$DGFF. We start with a bound on the left tail of $S_N^{*}-M_N^{*}$.
\begin{lemma}\label{lemma:up_bound_SN_left_tail}
	There exist constants $C,c>0$, such that, for all $N\in \mathbbm{N},$ and $0\leq \lambda\leq (\log \log N)^{\sfrac{2}{3}}$,
	\begin{align}
	\mathbbm{P}\left(\max_{v\in V_N}S_v^N \leq M_N^{*}-\lambda\right)\leq C e^{-c\lambda}.
	\end{align}
\end{lemma}
\begin{proof}
	By \autoref{lemma:right_tail_lb_MIBRW}, there are $\beta>0$ and $\delta_0 \in (0,1)$ such that, for all $N\in \mathbbm{N}$,
	\begin{align}\label{equation:lb_SN_delta}
	\mathbbm{P}\left(\max_{v \in V_N}S_v^N\geq m_N/\sqrt{\log(2)}-\beta\right)\geq \delta_0.
	\end{align}
	In particular, there is a $\kappa>0$ such that, for all $N\geq N^{'} \geq 4$,
	\begin{align} \label{eq:8.70}
	2\sqrt{\log(2)}\mathcal{I}_{\bar{\sigma}}(1)\log\left(\frac{N}{N^{'}}\right)- \frac{3}{4\sqrt{\log(2)}} \sum_{j=1}^{m} \bar{\sigma}_j \log\left(\log\left(\frac{N}{N^{'}}\right)\right)- \kappa &\leq M_N^{*}-M_{N^{'}}^{*}\\ &\leq 2\sqrt{\log(2)}\mathcal{I}_{\bar{\sigma}}(1)\log\left(\frac{N}{N^{'}}\right)+\kappa.\nonumber
	\end{align}
	We now pick $\lambda^{'}=\frac{\lambda}{2}$, $N^{'}=N\exp\left[-\frac{1}{2\sqrt{\log(2
			)}\mathcal{I}_{\bar{\sigma}}(1)}(\lambda^{'} - \beta -\kappa-4)\right]$ and set $n^\prime=\log_2 N^\prime$. With this choice, we deduce from \eqref{eq:8.70} that $M_N-M_{N^{'}}\leq \lambda^{'}-\beta.$ We divide $V_N$ into disjoint boxes by placing at each position $(3iN^{'},3jN^{'})$ a box of size $N^{'}$, for $1\leq i,j \leq \frac{N}{N^{'}}$. We call this collection of boxes $\mathcal{B}$ and note that the pairwise distances between two boxes are at least $2N^{'}$. This implies independence of the processes $\{S^{N^{'}}_v\}_{v\in B}$ on pairwise disjoint boxes. This allows us to bound the number of boxes $B\in \mathcal{B}$ from below by
		\begin{align}\label{equation:lb_B}
			\frac{N}{3 N^{\prime}}\geq \frac{1}{3}\exp\left[\frac{1}{2\sqrt{\log(2)} \mathcal{I}_{\bar{\sigma}}(1)}(\lambda^{'}-\beta -\kappa-4)\right].
		\end{align}
	Let  $\tilde{S}_v^N=S_v^{N^{'}}+X$, for $v\in B$ and $B\in \mathcal{B}$, where $X \sim \mathcal{N}(0, s^2)$ is an independent random variable and with $s^2$ such that $\mathrm{Var}(S_v^N)=\mathrm{Var}(\tilde{S}_v^{N})$.
	For $u,v \in \bigcup_{B \in \mathcal{B}}B$, we then have
	\begin{align}
	&\mathbbm{E}\left[(\tilde{S}_u^N-\tilde{S}_v^N)^2\right]=\mathbbm{E}\left[(S_u^{N^{'}}-S_v^{N^{'}})^2\right]\leq \mathbbm{E}\left[(S_u^N-S_v^N)^2\right].
	\end{align}
	An application of Slepian's Lemma gives that, for any $t\in \mathbbm{R}$,
	\begin{align}\label{equation:5.51}
	\mathbbm{P}\left(\max_{v \in V_N}S_v^N\leq t\right)\leq \mathbbm{P}\left(\max_{v \in \cup_{B \in \mathcal{B}}B}S_v^N\leq t\right)\leq \mathbbm{P}\left(\max_{v \in \cup_{B \in \mathcal{B}}B}\tilde{S}_v^N \leq t\right).
	\end{align}
	Using $M_N^{*}-\lambda^{'}\leq  M_{N^{'}}^{*}-\beta$ and \eqref{equation:lb_SN_delta}, one obtains, for each $B \in \mathcal{B}$,
	\begin{align}\label{equation:5.48}
	\mathbbm{P}\left(\max_{v \in B}S_v^{N^{'}}\geq M_N^{*}-\lambda^{'}\right)\geq \mathbbm{P}\left(\max_{v \in B}S_v^{N^{'}}\geq M_{N^{'}}^{*}-\beta\right)\geq \delta_0.
	\end{align}
	By \eqref{equation:5.48} and the independence of $\{S^{N^\prime}_v\}_{v\in B}$ and $\{S^{N^\prime}_v\}_{v\in B^\prime}$, for different $B, B^\prime \in \mathcal{B}$,
	\begin{align}\label{equation:5.50}
	\mathbbm{P}\left(\max_{v \in \cup_{B \in \mathcal{B}}B}S_v^{N^{'}}<M_N^{*}-\lambda^{'}\right)\leq (1-\delta_0)^{|\mathcal{B}|}.
	\end{align}
	As $\delta_0 \in (0,1)$, by \eqref{equation:lb_B}, there are constants, $C,c>0,$ such that
	\begin{align}
		(1-\delta_0)^{|\mathcal{B}|}\leq \exp\left[\frac{\log(1-\delta_0)}{3}\exp\left(\frac{1}{2\sqrt{\log(2)} \mathcal{I}_{\bar{\sigma}}(1)}(\lambda^{'}-\beta -\kappa-4)\right)\right]   \leq C e^{-c\lambda^{'}}. \label{eq:8.80}
	\end{align}
	Using \eqref{equation:5.51}, we can bound $\mathbbm{P}\left(\max_{v \in V_N}S_v^N\leq M^{*}_N-\lambda\right)$ from above by
	\begin{align}
	\mathbbm{P}\left(\max_{v \in \cup_{B \in \mathcal{B}}B}S_v^{N^{'}}<M_N^{*}-\lambda^{'}\right)+\mathbbm{P}(\theta \leq - \lambda^{'}) \leq C e^{-c\lambda^{'}},
	\end{align}
	where the last bound follows from \eqref{eq:8.80} and a Gaussian tail bound.
\end{proof}
 \autoref{lemma:up_bound_SN_left_tail} allows us to deduce the upper bound on the left tail of the centred maximum.
\begin{lemma}\label{lemma:up_bound_psiN_left_tail}
There exist constants, $C,c>0$, so that, for all $N\in \mathbbm{N},$ and $0\leq \lambda \leq (\log\log N)^{\sfrac{2}{3}}$,
	\begin{align}\label{equation:5.56}
		\mathbbm{P}\left(\max_{v \in V_N}\psi_v^N\leq m_N-\lambda\right)\leq Ce^{-c\lambda}.
	\end{align}
\end{lemma}
	\begin{proof}
	Following the proof of \autoref{lemma:5.2}, we see that, instead of $\mathbbm{P}\left(\max_{v \in V_N}\psi_v^N\leq  m_N-\lambda\right)$, it suffices to bound $\mathbb{P}\left(\max_{v \in 2^\kappa V_N} \psi_{v+(2^{\kappa+1}N,2^{\kappa+1}N)}^{2^{\kappa+2}N}\leq m_N-\lambda \right)$.
	By \autoref{lemma:cov_comp} $iv.$, there is a constant $\kappa_0>0$, such that, for all $\kappa\geq \kappa_0$,
	\begin{align}
	\mathrm{Var}\left[\psi_{2^{\kappa}v+(2^{\kappa+1}N,2^{\kappa+1}N)}^{2^{\kappa+2}N}\right]\leq \log(2) \mathrm{Var}\left[S_v^{2^{2\kappa}N}\right],\quad\forall v \in V_N.
	\end{align}
	Therefore, we can choose a collection of positive numbers, $\{a_v: \,v \in V_N\},$ and an independent standard Gaussian random variable, $X,$ so that, for any $N$ and $u,v \in V_N$,
	\begin{align}\label{equation:8.89}
	\mathrm{Var}\left[\psi_{2^{\kappa}v+(2^{\kappa +1}N,2^{\kappa +1}N)}^{2^{\kappa+2}N}+a_v X\right]= \log(2) \mathrm{Var}\left[S_v^{2^{2\kappa}N}\right],\quad \forall v \in V_N.
	\end{align} 
	As $\mathrm{Var}\left[S_v^{2^{2\kappa}N}\right]=\mathrm{Var}\left[S_v^{2^{2\kappa}N}\right]$, for all $v,w\in V_{2^{2\kappa}N}$, and by the uniform bound in \autoref{lemma:cov_comp} $ii.$, there is a constant $C_1>0$, such that
	\begin{align}
	|a_u-a_v|\leq C_1.
	\end{align}
	Writing $\tilde{u}=2^{\kappa}u+(2^{\kappa+1}N,2^{\kappa+1}N)$ and using \autoref{lemma:cov_comp} $ii$ and $iv.$, we get
	\begin{align} \label{equation:8.97}
		\mathbbm{E}\left[\psi_{\tilde{u}}^{2^{\kappa+2}N}\psi_{\tilde{v}}^{2^{\kappa+2}N}\right]\geq& \log(2)(n+\kappa) \mathcal{I}_{\sigma^2}\left(\frac{n+\kappa-\log_{+}\|2^\kappa u-2^\kappa v\|_2}{n+\kappa}\right) -c\nonumber
		\\ =&\log(2)(n+\kappa)\mathcal{I}_{\sigma^2}\left(\frac{n-\log_{+}\|u-v\|_2}{n+\kappa}\right) -c,
	\end{align}
	where $c>0$ is a constant.
	Further, taking into account that the Euclidean distance on the torus is bounded by the usual Euclidean distance, we have by \autoref{lemma:cov_comp} $ii.$,
	\begin{align}\label{equation:8.98}
		\mathbbm{E}\left[S_u^{2^{\kappa}N}S_v^{2^{\kappa}N}\right]\leq & (n+2\kappa)\mathcal{I}_{\sigma^2}\left(\frac{n+2\kappa-\log_{+}\|u-v\|_2}{n+2\kappa}\right) +C,
	\end{align}
	where $C>0$ is another constant.
	Comparing \eqref{equation:8.97} and \eqref{equation:8.98}, one deduces, using \eqref{equation:8.89} that there is a $\kappa_0$, such that, for $\kappa\geq \kappa_0$,
	\begin{align}
	\mathbbm{E}\left[\left(\psi_{2^{\kappa}u+(2^{\kappa +1}N,2^{\kappa +1}N)}^{2^{\kappa+2}N}+a_u X\right)\left(\psi_{2^{\kappa}v+(2^{\kappa +1}N,2^{\kappa +1}N)}^{2^{\kappa+2}N}+a_v X\right)\right] \leq \log(2) \mathbbm{E}\left[S_u^{2^{\kappa}N}S_v^{2^{\kappa}N}\right]. \label{equation:8.93}
	\end{align}
	Using \eqref{equation:8.93} and \eqref{equation:8.89}, we can apply Slepian's lemma to obtain
	\begin{align}\label{equation:5.63}
	&\,\mathbbm{P}\left(\max_{v\in V_N} \psi_{2^{\kappa}v+(2^{\kappa +1}N,2^{\kappa +1}N)}^{2^{\kappa+2}N}\leq m_N-\lambda\right)\nonumber\\\leq &\,\mathbbm{P}\left(\max_{v \in V_N} \psi_{2^{\kappa}v+(2^{\kappa +1}N,2^{\kappa +1}N)}^{2^{\kappa+2}N} +a_v X\leq m_N-\frac{\lambda}{2}\right)+ \mathbbm{P}\left(X\leq - \frac{\lambda}{C_{\kappa}}\right)\nonumber\\\leq &\,\mathbbm{P}\left(\max_{v \in V_N}S_v^{2^{2\kappa}N}\leq M_N^{*}-\frac{\lambda}{2\sqrt{\log(2)}}\right)+ \mathbbm{P}\left(X\leq - \frac{\lambda}{C_{\kappa}}\right),
	\end{align}
	where $C_{\kappa}>0$ is a constant that solely depends on $\kappa$. Note that there is a collection of boxes $\mathcal{V}$, consisting of at most $2^{8 \kappa}$ translated copies of $V_N$, such that $V_{2^{2\kappa}N}\subset \cup_{V \in \mathcal{V}} V$. Since
	\begin{equation}
		\left\{\max_{v \in V_{2^{2\kappa}N}} S^{2^{\kappa}N}_v \leq M_N^{*}- x \right\}= \cap_{V \in \mathcal{V}} \left\{\max_{v \in V_N} S^{2^{\kappa}N}_v\leq M_N^{*}-x \right\},
	\end{equation}  
	we have, by the FKG inequality \cite[Proposition~1]{fortuin1971}, that
	\begin{align}\label{equation:5.66}
		\mathbbm{P}\left(\max_{v \in V_{2^{2\kappa}N}}S_v^{2^{2\kappa}N}\leq M_N^{*}-\frac{\lambda}{2\sqrt{\log(2)}}\right)  \geq \left(\mathbbm{P}\left(\max_{v \in V_N}S_v^{2^{2\kappa}N}\leq M_N^{*}-\frac{\lambda}{2\sqrt{\log(2)}}\right)\right)^{8 \kappa}.
	\end{align}
	Using \eqref{equation:5.66} and then \autoref{lemma:up_bound_SN_left_tail}, we bound \eqref{equation:5.63} from above by
	\begin{align}
		\mathbbm{P}\left(\max_{v\in V_{2^{\kappa}N}} \psi_{2^{\kappa}v+(2^{\kappa +1}N,2^{\kappa +1}N)}^{2^{\kappa+2}N}\leq m_N-\lambda\right) &\leq  \mathbbm{P}\left(\max_{v\in V_N} \psi_{2^{\kappa}v+(2^{\kappa +1}N,2^{\kappa +1}N)}^{2^{\kappa+2}N}\leq m_N-\lambda\right)\nonumber \\
		&  \leq \mathbbm{P}\left(\max_{v \in V_N}S_v^{2^{2\kappa}N}\leq M_N^{*}-\frac{\lambda}{2\sqrt{\log(2)}}\right)+ \mathbbm{P}\left(X\leq - \frac{\lambda}{C_{\kappa}}\right) \nonumber\\ 
		&\leq \left(\mathbbm{P}\left(\max_{v \in V_{2^{2\kappa}N}}S_v^{2^{2\kappa}N}\leq M_N^{*}-\frac{\lambda}{2\sqrt{\log(2)}}\right)\right)^{1/ (8\kappa)} \nonumber\\&\quad + \mathbbm{P}\left(X\leq - \frac{\lambda}{C_{\kappa}}\right) \leq \tilde{C}e^{-\tilde{c}\lambda},
	\end{align}
	where $\tilde{C},\tilde{c}>0$ are constants that are independent of $N$. This concludes the proof of \autoref{lemma:up_bound_psiN_left_tail}.
	\end{proof}
We now have all the ingredients to finish the proof of \autoref{thm1}.
\begin{proof}[Proof of \autoref{thm1}]
	The upper bound on the right-tail in \eqref{eq:thm1_right_tail} follows using \autoref{proposition:up_bound_right_tail}. A combination of \autoref{lemma:5.3} with \autoref{lemma:right_tail_lb_MIBRW} implies the lower bound on the right-tail in \eqref{eq:thm1_right_tail}.
	The second statement, the upper bound for the left tail \eqref{eq:thm1_left_tail}, is given by \autoref{lemma:up_bound_psiN_left_tail}, which finishes the proof.
\end{proof}

%\clearpage
%\section{Summary and open questions}
%\input{summary_openquestions.tex}
\appendix
\section{Gaussian comparison}\label{appendix:A}
\begin{thm}[Borell's inequality, {\cite[Lemma~3.1]{MR2814399}}]\label{theorem:borell}
	Let $T$ be compact and $\{ X_t \}_{t\in T}$ a centred Gaussian process on $T$ with continuous covariance. Further assume that almost surely, $X^* \coloneqq \sup_{t \in T} X_t < \infty$. Then, 
	\begin{align}
	&\mathbbm{E}[X^*]< \infty,
	\end{align}
	and
	\begin{align}
	&\mathbbm{P}\left(\left|X^*-\mathbbm{E}[X^*]\right|>x\right)\leq 2 e^{-\sfrac{x^2}{2 \sigma_T^2}},
	\end{align}
	where $\sigma_T^2 \coloneqq \max_{t \in T} \mathbbm{E}[X_t^2]$.
\end{thm}
\begin{thm}[Slepian's Lemma, {\cite[Theorem~3.11]{MR2814399}}]\label{thm:slepian}
	Let $T= \{1,\dotsc,n \}$ and $X,Y$ be two centred Gaussian vectors. Assume that we have two subsets $A,B \subset T \times T$ satisfying
	\begin{align}
		& \mathbbm{E}[X_i X_j] \leq \mathbbm{E}[Y_i Y_j], \, \quad(i,j) \in A \\
		& \mathbbm{E}[X_i X_j] \geq \mathbbm{E}[Y_i Y_j], \, \quad(i,j) \in B \\
		& \mathbbm{E}[X_i X_j] = \mathbbm{E}[Y_i Y_j], \, \quad (i,j)\notin A \cup B.
	\end{align}
	Further, suppose that $f: \mathbbm{R}^n \rightarrow \mathbbm{R}$ is a smooth function with at most exponential growth at infinity of $f$ itself, as well as its first and second derivatives, and that
	\begin{align}
		& \partial_{ij} f \geq 0, \, \quad(i,j) \in A \\
		& \partial_{ij} f \leq 0, \, \quad(i,j) \in B.
	\end{align}
	Then, 
	\begin{align}
		\mathbbm{E}[f(X)]\leq \mathbbm{E}[f(Y)].
	\end{align}
\end{thm}
We use Slepian's Lemma in a particular setting, i.e. we assume that $\mathbbm{E}\left[X_i^2\right]=\mathbbm{E}\left[Y_i^2\right]$ and $\mathbbm{E}\left[X_iX_j\right]\geq \mathbbm{E}\left[Y_iY_j\right]$ for all $i,j \in T.$ We then have for any $x \in \mathbbm{R}$,
	\begin{align}
		\mathbbm{P}\left(\max_{i \in T}X_i >x\right)\leq \mathbbm{P}\left(\max_{i \in T}Y_i>x\right).
	\end{align}
In particular, $\mathbbm{E}\left[\max_{i \in T}X_i\right]\leq \mathbbm{E}\left[\max_{i\in T}Y_i\right].$ 
If we only want to compare the expectation of maxima we do not need the equality of variances. This is a result due to Sudakov and Fernique.
\begin{thm}[Sudakov-Fernique, {\cite{MR0413238}}]\label{thm:sudakov}
	Let $I$ be an arbitrary set of finite size $n,$ $\{X_i\}_{i \in I}, \{Y_i\}_{i \in I}$ be two centred Gaussian vectors. Define $\gamma_{ij}^{X} \coloneqq \mathbbm{E}[(X_i-X_j)^2]$, $\gamma_{ij}^{Y} \coloneqq \mathbbm{E}[(Y_i-Y_j)^2]$. Let $\gamma \coloneqq \max_{i,j} |\gamma_{ij}^X - \gamma_{ij}^Y|.$ Then,
	\begin{align}
		 \left|\mathbbm{E}[X^*] - \mathbbm{E}[Y^*]\right| \leq \sqrt{\gamma \log(n)}.
	\end{align}
	If $\gamma_{ij}^X \leq \gamma_{ij}^Y$ for any $i,j\in I$, then
	\begin{align}
		\mathbbm{E}[X^*]\leq \mathbbm{E}[Y^*].
	\end{align}
\end{thm}
In particular, if $\{X_i\}_{i \in I} \text{ and } \{Y_i\}_{i \in I}$ are independent centred Gaussian fields without any additional assumptions on their correlations, one deduces
\begin{equation}
	\mathbbm{E}\left[ \max_{i \in I} (X_i + Y_i)\right] \geq \mathbbm{E}\left[\max_{i \in I} X_i \right].
\end{equation}
\section{Covariance estimates}\label{appendix:B}
For particles $v,w \in V_N$, let
\begin{align}
	b_N(v,w)\coloneqq \max\{\lambda \in [0,1]: [v]^N_{\lambda}\cap [w]^N_{\lambda}\neq \emptyset \}
\end{align}
denote the branching scale. The key point is that beyond $b_N(v,w),$ increments are independent, that is for $1\geq \lambda^\prime > \lambda> b_N(v,w)$, $\phi^N_v(\lambda^{\prime} )-\phi^N_w(\lambda)$ is independent of $\phi^N_w(\lambda^\prime)-\phi^N_w(\lambda)$, whereas increments before the branching scale are correlated.
Further, for some $B\subset V_N$, we set
\begin{align}
	\phi^N_v(B)\coloneqq \mathbb{E}\left[\phi^N_v | \sigma\left(\phi^N_w:w\in B^c\right)\right].
\end{align}
Recall that for $\lambda\in [0,1]$, we also write $\phi^N_v(\lambda)=\phi^N_v([v]_\lambda^N)$.
\begin{lemma_appendix}\label{lemma:cov_estimate_sd_dgff}
	Let $\delta\in (0,1/2)$ and $N\in \mathbb{N}$ such that $\min_{1\leq i \leq M}2^{\frac{2}{\Delta \lambda_{i}}}\leq N$, as well as $N^{\lambda_1}>\delta^{-1}$. Let $v,w\in V_N^\delta$ and assume that the branching scale $b_N(v,w)$ coincides with a scale parameter, i.e. $b_N(v,w)=\lambda_{i}$ for some $i\in \mathbb{N}$. Then for any $0\leq i,j\leq M$ with $\lambda_{i},\lambda_{j}\leq b_N(v,w)$, we have
	\begin{align}
	\mathbb{E}\left[\Delta\phi^N_v(\lambda_i)\Delta\phi^N_w(\lambda_j)\right]=\Delta \lambda_i \log(N) \mathbbm{1}_{i=j}+ O\left(1\right).
	\end{align}
	\begin{proof}
		For $v=w$ the statement is contained in \cite[Lemma A.2]{MR3541850}. Let us assume $v\neq w$ throughout the proof .
		We start with the case $i=j$.
		More, we assume $[v]_{\lambda_i}\cap [w]_{\lambda_i}\neq \emptyset$, i.e. the boxes should intersect at least at the boundary. If this is not the case, we can subdivide the scales further and use that beyond $b_N(v,w)$ the respective increments are independent. This implies that $\|v-w\|_2\leq \sqrt{2}N^{1-\lambda_i}$. We now pick a box $B$ of side length $2N^{1-\lambda_i}$, centred at the middle of the line connecting the vertices $v$ and $w$. This ensures the inclusion
		\begin{align}\label{equation:inclusion_B_1}
		\sigma\left(\phi^N_u:u\in B^c\right)\subset\sigma\left(\phi^N_u:u\in[v]^c_{\lambda_i}\right),\,\sigma\left(\phi^N_u:u\in [w]^c_{\lambda_i}\right).
		\end{align}
		Next we pick a box $\tilde{B}$ of side length $\frac{1}{2}N^{1-\lambda_{i-1}}$ with the same centre as $B$. For $N$ as in the assumption, this implies in particular that $\sigma\left(\phi^N_u:u\in\tilde{B}^c\right)\subset\sigma\left(\phi^N_u:u\in B^c\right)$, as well as 
		\begin{align}\label{equation:inclusion_B_2}
		\sigma\left(\phi^N_u:u\in[v]^c_{\lambda_{i-1}}\right),\,\sigma\left(\phi^N_u:u\in [w]^c_{\lambda_{i-1}}\right)\subset \sigma\left(\phi^N_u:u\in\tilde{B}^c\right).
		\end{align} 
		We write $\Delta \phi^N_v(B)=\phi^N_v(B)-\phi^N_v(\tilde{B})$ and compute,
		\begin{align}
		\mathbb{E}\left[\Delta \phi^N_v(\lambda_i)\Delta \phi^N_w(\lambda_i)\right]&= \mathbb{E}\left[\left(\phi^N_v(\lambda_i)-\phi^N_v(B)+\nabla\phi^N_v(B)+\phi^N_v(\tilde{B})-\phi^N_v(\lambda_{i-1})\right) \right.\nonumber\\ & \left. \quad \times \left(\phi^N_w(\lambda_i)-\phi^N_w(B)+\nabla\phi^N_w(B)+\phi^N_w(\tilde{B})-\phi^N_w(\lambda_{i-1})\right)\right]\nonumber\\
		& =\mathbb{E}\left[\Delta\phi^N_v(B)\Delta\phi^N_w(B)\right] \label{eq:pf_cov_est_line1}\\
		&\quad+ \mathbb{E}\left[\Delta \phi^N_v(B)\left(\phi^N_w(\lambda_i)-\phi^N_w(B)+\phi^N_w(\tilde{B})-\phi^N_w(\lambda_{i-1})\right)\right]\label{eq:pf_cov_est_line2}\\
		&\quad+\mathbb{E}\left[\left(\phi^N_v(\lambda_i)-\phi^N_v(B)\right)\left(\phi^N_w(\lambda_i)-\phi^N_w(B)+\phi^N_w(\tilde{B})-\phi^N_w(\lambda_{i-1})\right)\right]\label{eq:pf_cov_est_line3}\\
		&\quad-\mathbb{E}\left[\left(\phi^N_v(\lambda_{i-1})-\phi^N_v(\tilde{B})\right)\left(\phi^N_w(\lambda_i)-\phi^N_w(B)+\phi^N_w(\tilde{B})-\phi^N_w(\lambda_{i-1})\right)\right]\label{eq:pf_cov_est_line4}.			
		\end{align}
		Using the conditional covariance identity
		\begin{align}
		\mathbb{E}\left[\mathbb{E}\left[X|\mathcal{A}\right]\mathbb{E}\left[Y|\mathcal{A}\right]\right]=\mathbb{E}\left[XY\right]-\mathbb{E}\left[\left(X-\mathbb{E}\left[X|\mathcal{A}\right]\right)\left(Y-\mathbb{E}\left[Y|\mathcal{A}\right]\right)\right],
		\end{align}
		with $X=\phi^N_v(1)-\phi^N_v(\tilde{B}), \,Y=\phi^N_w(1)-\phi^N_w(\tilde{B})$ and $\mathcal{A}=\sigma\left(\phi^N_u:u\notin B^o\right)$, along with noting that by the Gibbs-Markov property of the DGFF $\phi^N_v(1)-\phi^N_v(\tilde{B})\overset{d}{=}\phi^{\tilde{B}}_v$, we can write the first term \eqref{eq:pf_cov_est_line1} as
		\begin{align}
		\mathbb{E}\left[\phi^B_v \phi^B_w\right]-\mathbb{E}\left[\phi^{\tilde{B}}_v\phi^{\tilde{B}}_w\right]&= \log\left(N^{1-\lambda_i+\log(2)/\log(N)}\right)- \log(\|v-w\|\vee 1)- \log\left(N^{1-\lambda_i-\log(2)/\log(N)}\right)\nonumber\\
		&\quad +\log(\|v-w\|\vee 1) +O(1)= \Delta\lambda_i \log(N) +O(1).
		\end{align}
		For the remaining terms we need to show that they are at most of constant order. As the last two terms  \eqref{eq:pf_cov_est_line3} and \eqref{eq:pf_cov_est_line4} can be estimated the same way, we only deal with \eqref{eq:pf_cov_est_line3}. Using Cauchy-Schwarz,
		\begin{align}
		&\mathbb{E}\left[\left(\phi^N_v(\lambda_i)-\phi^N_v(B)\right)\left(\phi^N_w(\lambda_i)-\phi^N_w(B)-\phi^N_w(\lambda_{i-1})+\phi^N_w(\tilde{B})\right)\right]\nonumber\\
		\leq& \mathbb{E}\left[\left(\phi^N_v(\lambda_i)-\phi^N_v(B)\right)^2\right]^{1/2}\left(\mathbb{E}\left[\left(\phi^N_w(\lambda_i)-\phi^N_w(B)\right)^2\right]^{1/2}+\mathbb{E}\left[\left(\phi^N_w(\tilde{B})-\phi^N_w(\lambda_{i-1})\right)^2\right]^{1/2}\right)\nonumber\\=&(\log(2)+c_1)(\log(2)+c_2+\log(2)+c_3)=O(1).
		\end{align}
		To estimate \eqref{eq:pf_cov_est_line2} we make exhaustive use of our choice of boxes and use the relations \eqref{equation:inclusion_B_1} and \eqref{equation:inclusion_B_2} along with the tower property for conditional expectations and the law of total expectation, i.e. we first observe that both $\mathbb{E}\left[\phi^N_v(B)\phi^N_w(\lambda_i)\right]=\mathbb{E}\left[\phi^N_v(B)\phi^N_w(B)\right]$ and $\mathbb{E}\left[\phi^N_v(\tilde{B})\phi^N_w(\lambda_i) \right]=\mathbb{E}\left[\phi^N_v(\tilde{B})\phi^N_w(\tilde{B}) \right]$ hold. Using this, we reformulate \eqref{eq:pf_cov_est_line2}, i.e.
		\begin{align}
		&\mathbb{E}\left[\Delta \phi^N_v(B)\left(\phi^N_w(\lambda_i)-\phi^N_w(B)+\phi^N_w(\tilde{B})-\phi^N_w(\lambda_{i-1})\right)\right]\nonumber\\ 
		=&\mathbb{E}\left[\phi^N_v(B)\left(\phi^N_w(\tilde{B})-\phi^N_w(\lambda_{i-1})\right)\right]-\mathbb{E}\left[\phi^N_v(\tilde{B})\left(\phi^N_w(\tilde{B})-\phi^N_w(\lambda_{i-1})\right)\right]\nonumber\\
		=&\mathbb{E}\left[\phi^N_v\left(\phi^N_w(\tilde{B})-\phi^N_w(\lambda_{i-1})\right)\right]-\mathbb{E}\left[\phi^N_v\left(\phi^N_w(\tilde{B})-\phi^N_w(\lambda_{i-1})\right)\right]=0.
		\end{align}
		For the remaining case $i\neq j$, we note that for $|i-j|\geq 2$ increments are independent as the difference of the boxes do not intersect for any $v,w\in V_N$, as we assume $N$ to be sufficiently large. The only remaining case is $j=i-1$. Note that in this case, the increment $\Delta\phi^N_v(\lambda_{i})$ is independent of the increment $\phi^N_w(\lambda_{i-1}-\frac{\log(4)}{\log(N)})-\phi^N_w(\lambda_{i-2})$, as the annuli of the corresponding boxes do not intersect. This gives,
		\begin{align}\label{equation:7.28}
		\mathbb{E}\left[\Delta \phi^N_v(\lambda_{i})\Delta\phi^N_w(\lambda_{i-1})\right]=&\mathbb{E}\left[\Delta\phi^N_v(\lambda_{i})\left(\phi^N_w(\lambda_{i-1})-\phi^N_w\left(\lambda_{i-1}-\frac{\log(4)}{\log(N)}\right)+\phi^N_w\left(\lambda_{i-1}-\frac{\log(4)}{\log(N)}\right)-\phi^N_w(\lambda_{i-2})\right)\right]\nonumber\\
		=&\mathbb{E}\left[\Delta \phi^N_v(\lambda_{i})\left(\phi^N_w(\lambda_{i-1})-\phi^N_w\left(\lambda_{i-1}-\frac{\log(4)}{\log(N)}\right)\right)\right]
		\nonumber\\
		=&\mathbb{E}\left[\left(\phi^N_v(\lambda_{i})-\phi^N_v([w]_{\lambda_{i}})+\phi^N_v([w]_{\lambda_{i}})-\phi^N_v([w]_{\lambda_{i-1}})+\phi^N_v([w]_{\lambda_{i-1}})-\phi^N_v(\lambda_{i-1})\right)\right.\nonumber\\
		&\times\left. \left(\phi^N_w(\lambda_{i-1})-\phi^N_w\left(\lambda_{i-1}-\frac{\log(4)}{\log(N)}\right)\right)\right].
		\end{align}
		Provided $N$ is large, we have $[v]_{\lambda_{i}}^c\cup [w]_{\lambda_{i}}^c \supset[w]_{\lambda_{i-1}}^c \supset [w]_{\lambda_{i-1}-\frac{\log(4)}{\log(N)}}^c$ and so by the tower property and the law of total expectation, we deduce
		\begin{align}
		\mathbb{E}&\left[\left(\phi^N_v(\lambda_{i})-\phi^N_v([w]_{\lambda_{i}})\right)\left(\phi^N_w(\lambda_{i-1})-\phi^N_w\left(\lambda_{i-1}-\frac{\log(4)}{\log(N)}\right)\right)\right]\nonumber\\
		&=\mathbb{E}\left[\mathbb{E}\left[\phi^N_v\left(\phi^N_w(\lambda_{i-1})-\phi^N_w\left(\lambda_{i-1}-\frac{\log(4)}{\log(N)}\right)\right)\bigg|\sigma\left(\phi^N_u: u\in [v]_{\lambda_{i}}^c\right)\right]\right]\nonumber\\
		&\quad-\mathbb{E}\left[\mathbb{E}\left[\phi^N_v\left(\phi^N_w(\lambda_{i-1})-\phi^N_w\left(\lambda_{i-1}-\frac{\log(4)}{\log(N)}\right)\right)\bigg|\sigma\left(\phi^N_u: u\in [w]_{\lambda_{i}}^c\right)\right] \right]=0.			
		\end{align}
		As the annuli $[w]_{\lambda_{i-1}}\backslash [w]_{\lambda_{i}}$ and $[w]_{\lambda_{i-1}-\frac{\log(4)}{\log(N)}}\backslash [w]_{\lambda_{i-1}}$ do not intersect, we have independence of the corresponding increments, i.e.
		\begin{align}
		\mathbb{E}\left[\left(\phi^N_v([w]_{\lambda_{i}})-\phi^N_v([w]_{\lambda_{i-1}})\right)\left(\phi^N_w(\lambda_{i-1})-\phi^N_w\left(\lambda_{i-1}-\frac{\log(4)}{\log(N)}\right)\right)\right]=0.
		\end{align}
		The remaining term in \eqref{equation:7.28} can be bounded in a first step by the Cauchy-Schwarz inequality,
		\begin{align}
		&\mathbb{E}\left[\left(\phi^N_v([w]_{\lambda_{i-1}})-\phi^N_v(\lambda_{i-1})\right)\left(\phi^N_w(\lambda_{i-1})-\phi^N_w\left(\lambda_{i-1}-\log(4)/\log(N)\right)\right)\right]\nonumber\\&\quad
		\leq c \sqrt{\log(4)} \mathbb{E}\left[\left(\phi^N_v([w]_{\lambda_{i-1}})-\phi^N_v(\lambda_{i-1})\right)^2\right]^{1/2}.
		\end{align}
		In order to bound the expectation on the right hand side, we consider a box $B$ centred at the middle of the line connecting $v$ and $w$ of side length $N^{1-\lambda_{i-1}}-\sqrt{2}N^{1-\lambda_{i}}$. The assumption $\|v-w\|_\infty\leq \sqrt{2}N^{1-\lambda_{i}}$ ensures the inclusion
		$B\subset [v]_{\lambda_{i-1}}\cap [w]_{\lambda_{i-1}}$. This allows us to compute in a similar fashion as in the first case \eqref{eq:pf_cov_est_line1}, i.e.
		\begin{align}
		\mathbb{E}\left[\left(\phi^N_v([w]_{\lambda_{i-1}})-\phi^N_v(\lambda_{i-1})\right)^2\right]&= 	\mathbb{E}\left[\left(\phi^N_v([w]_{\lambda_{i-1}})-\phi^N_v(B)+\phi^N_v(B)-\phi^N_v(\lambda_{i-1})\right)^2\right]\nonumber\\
		&\leq 4\max\left(\mathbb{E}\left[\left(\phi^N_v([w]_{\lambda_{i-1}})-\phi^N_v(B)\right)^2\right],\mathbb{E}\left[\left(\phi^N_v(B)-\phi^N_v(\lambda_{i-1})\right)^2\right] \right)\nonumber\\
		&\leq 4(c+\log(N^{1-\lambda_{i-1}})-\log(N^{1-\lambda_{i-1}}(1-\sqrt{2}N^{-\Delta \lambda_{i}})))\leq C.
		\end{align}
		The constants $c,C>0$ can be chosen uniformly in $N$, however depending on the scale parameters.
		Altogether, we obtain
		\begin{align}
		\mathbb{E}\left[\Delta \phi^N_v(\lambda_{i})\Delta \phi^N_w(\lambda_{j})\right]\leq C,
		\end{align}
		for some constant $C>0$ that is uniform in $N$, which finishes the proof.
	\end{proof}
\end{lemma_appendix}
\begin{proof}[Proof of \autoref{lemma:cov_comp}]
		For a proof of the statements $i.$ and $iii.$, we refer to \cite[Lemma 2.2]{MR2846636}. We have that $\log_{+}(d^N_\infty(v,w))\leq \log_{+}(d^N(v,w))\leq \log_{+}(d_\infty^N(v,w))+1$.
		We begin with the proof of the second statement. Note that if $1 \leq k< \log_{+}(d_{\infty}^N(v,w)+1),$ there are no boxes of size $2^k$ that cover both $v$ and $w$. Thus, if $B, \tilde{B}$ are boxes such that one covers $v$ but not $w$ and the other $w$ but not $v$, the associated random variables $b_{k,B}, b_{k,\tilde{B}}$ are independent. And so, only random variables $b_{k,B}$ associated to boxes of size $2^k$ with $k> \ceil*{\log_2(d_{\infty}^N(v,w)+1)}$ contribute to the covariance. For $v=(v_1,v_2), \, w=(w_1,w_2)$ and $i=1,2$, we write $
		r_i (v,w)= \min(|v_i - w_i|,|v_i-w_i-N|,|v_i-w_i+N|)$. Using the fact that the number of common boxes for $v,w\in V_N$ is given by $[2^k-r_1(v,w)][2^k-r_2(v,w)]$,
		\begin{align}
			\mathbb{E}\left[S^N_v S^N_w\right]&=\sum_{k= \ceil*{\log_{+}(d_{\infty}^N(v,w))}}^{n} 2^{-2k} \sigma^2\left(\frac{n-k}{n}\right)[2^k-r_1(v,w)][2^k-r_2(v,w)] \\ 
			&=\sum_{k= \ceil*{\log_{+}(d_{\infty}^N(v,w))}}^{n}\left[\left(1- \frac{r_1(v,w)}{2^k} - \frac{r_2 (v,w)}{2^k} + \frac{r_1(v,w) r_2(v,w)}{2^{2k}}\right)\left(\sum_{i=1}^{M}\mathbbm{1}_{n-k \in (\lambda_{i-1} n,\lambda_i n ] }\sigma_i^2\right)\right].\nonumber
		\end{align}
		We note that since $a+b-ab\geq 0$ for $0 \leq a,b \leq 1,$ we get
		\begin{align}
			\mathbb{E}\left[S^N_v S^N_w\right]  &\leq
			n \sum_{i=1}^{M} \sigma_i^2 \Delta \lambda_i - \sum_{i=1}^{M} \sigma_i^2 [n\Delta \lambda_i \mathbbm{1}_{n-\ceil*{\log_{+}(d_{\infty}^N (v,w))}\leq\lambda_i n} \nonumber\\&\quad+ \left[\lambda_i n -\left(n-\ceil*{\log_{+}(d_\infty^N(v,w))}\right)\right] \mathbbm{1}_{\lambda_{i-1}n< n -\ceil*{\log_{+}(d_\infty^N(v,w))}<\lambda_i n}] \nonumber \\&= 2\sum_{i=1}^{M}\sigma_i^2+
			\sum_{i=1}^{M} \sigma_i^2[n\Delta \lambda_i \mathbbm{1}_{ n-\ceil*{\log_{+}(d^N(v,w))} \geq \lambda_i n} \nonumber\\ &\quad +\left((1-\lambda_{i-1})n\-\ceil*{\log_{+}(d^N(v,w))}\right) \mathbbm{1}_{\lambda_{i-1}n<n-\ceil*{\log_{+}(d^N (v,w))}< \lambda_i n}]\nonumber\\
			&=2\mathcal{I}_{\sigma^2}(1)+n \mathcal{I}_{\sigma^2}\left(\frac{n-\ceil*{\log_{+}\left(d^N(v,w)\right)}}{n}\right).
		\end{align}
		On the other hand, since $a+b-ab\leq a+b$ for $a,b\geq 0$, we get
		\begin{align}
			\mathbb{E}\left[S^N_v S^N_w\right] & \geq \sum_{k=\ceil*{\log_{+}(d_{\infty}^N(v,w))}}^{n} \sigma^2\left(\frac{n-k}{n}\right) - \max_{1\leq i \leq M}\sigma^2\left(\frac{i}{n}\right) 2^{-k+1} d_\infty^N(v,w) \nonumber \\ & \geq  \sum_{i=1}^{M} \sigma_i^2 \left[n\Delta \lambda_i \mathbbm{1}_{ n-\ceil*{\log_{+}(d^N(v,w))} \geq \lambda_i n} + \left((1-\lambda_{i-1})n\right. \right.\nonumber\\ &\quad \left. \left. - \ceil*{\log_{+}(d^N(v,w))}\right) \mathbbm{1}_{\lambda_{i-1}n<n-\ceil*{\log_{+}(d^N(v,w))}< \lambda_i n} \right] - C \nonumber\\
			&=n \mathcal{I}_{\sigma^2}\left(\frac{n-\ceil*{\log_{+}\left(d^N(v,w)\right)}}{n} \right)-C,
		\end{align}
		where in the second step we did a rescaling from $[0,n]$ onto the unit interval $[0,1]$ and where $C>0$ is a constant independent of $N$ with $C>2\max_{1\leq i \leq M} \sigma^2(i/M)$ that deals with the second part of the sum.\\
		To prove the last statement $iv.$, we note that beyond the branching scale, $N$ being sufficiently large (see assumptions of \autoref{lemma:cov_estimate_sd_dgff}) and by the Gibbs-Markov property, increments are independent as the annuli of the corresponding boxes do not intersect (see for instance \cite[Section 2]{MR3541850}). Moreover, by a refinement of the scale parameters and possibly allowing for an additional uniformly bounded constant, we can assume that the branching scale coincides with a scale parameter. With this we can apply \autoref{lemma:cov_estimate_sd_dgff} and obtain the result, i.e.
		\begin{align}
			\mathbb{E}\left[\psi^{4N}_x \psi^{4N}_y\right]  &=  \mathbbm{E}\left[\sum_{i=1}^{M}\sum_{j=1}^{M} \sigma_i \sigma_j \Delta \phi_x^{4N}(\lambda_i) \Delta \phi_y^{4N}(\lambda_j)\right] = \sum_{i=1}^{M} \sigma_i^2 \mathbbm{E}\left[(\Delta \phi_x^{4N} (\lambda_i))^2 \mathbbm{1}_{n-\ceil*{\log_{+}(\|v-w\|_2)} \geq \lambda_i} \right. \nonumber \\&\quad \left.+\left(\phi_x^{4N}\left(\frac{n-\ceil*{\log_{+}(\|x-y\|_2)}}{n}\right)-\phi_x^{4N}(\lambda_{i-1})\right) \mathbbm{1}_{\lambda_{i-1}n<n-\ceil*{\log_{+}(\|x-y\|_2)}< \lambda_i n}\right] +O(1)\nonumber \\
			&=\log(2)\sum_{i=1}^{M} \sigma_i^2\left[n \Delta \lambda_{i} \mathbbm{1}_{n-\ceil*{\log_{+}\|x-y\|_2}\geq \lambda_{i} n} +((1-\lambda_{i-1})n\right.\nonumber\\ &\quad \left. -\ceil*{\log_{+}\|x-y\|_2} \mathbbm{1}_{\lambda_{i-1}n<n-\ceil*{\log_{+}\|x-y\|_2}<\lambda_{i}n})\right] + O(1)
			\nonumber\\ &= \log(2)n \mathcal{I}_{\sigma^2}\left(\frac{n-\ceil*{\log_{+}\|x-y\|_2}}{n}\right) + O(1),
		\end{align}
		where $O(1)$ is uniform in N.
	\end{proof}

\bibliography{literature.bib}
\bibliographystyle{abbrv}
%\printbibliography
\end{document}